\documentclass[11pt]{amsart}
\usepackage{amsmath,amsthm,amsfonts,amssymb,mathrsfs,pb-diagram,amscd}
\usepackage{amsmath}
\usepackage{amsxtra}
\usepackage{amscd}
\usepackage{amsthm}
\usepackage{amsfonts}
\usepackage{amssymb}
\usepackage{eucal}
\usepackage{color}
\usepackage[enableskew,vcentermath]{youngtab}

\DeclareMathOperator\Std{Std}
\DeclareMathOperator\res{res}
\allowdisplaybreaks[4]

\newtheorem{thm}{Theorem}[section]
\theoremstyle{plain}

\newtheorem{lem}[thm]{Lemma}

\newtheorem{prop}[thm]{Proposition}
\newtheorem{cor}[thm]{Corollary}

\theoremstyle{definition}
\newtheorem{defn}[thm]{Definition}
\newtheorem{example}[thm]{Example}

\theoremstyle{remark}
\newtheorem{rem}[thm]{Remark}
\newtheorem{conjecture}[thm]{Conjecture}

\definecolor{A}{rgb}{.75,1,.75}
\def\Sym{\mathfrak{S}}

\numberwithin{equation}{section}

\newcommand{\Z}{\mathbb Z}
\newcommand{\N}{\mathbb N}

\newcommand{\Cl}{\mathcal C}

\newcommand{\mHcn}{\mathcal{H}_{\Delta}(n)} 
\newcommand{\mHfcn}{\mathcal{H}^f_{\Delta}(n)} 
\newcommand{\mhcn}{\mathfrak{H}_{\Delta}(n)} 
\newcommand{\mhgcn}{\mathfrak{H}^g_{\Delta}(n)} 
\newcommand{\undla}{\underline{\lambda}}
\newcommand{\HC}{\mathcal{H}(n)}
\newcommand{\undQ}{\underline{Q}}

{\vskip-\lastskip\medskip
  \noindent
  {\em #1.}\enspace
  }%
{\qed\par\medskip
  }

\begin{document}

\title[Affine Hecke-Clifford algebras]{On  representation theory of cyclotomic Hecke-Clifford algebras}
	\author{Lei Shi}\address{Max-Planck-Institut f\"Ur Mathematik\\
	Vivatsgasse 7, 53111 Bonn\\
	Germany}
\email{leishi@mpim-bonn.mpg.de}

\author{Jinkui Wan}\address{School of Mathematics and Statistics\\
  Beijing Institute of Technology\\
  Beijing, 100081, P.R. China}
\email{wjk302@hotmail.com}

\begin{abstract}
 In this article, we give an explicit construction of the simple modules  for both non-degenerate and degenerate cyclotomic Hecke-Clifford superalgebras over an algebraically closed field of characteristic not equal to $2$ under certain condition in terms of parameters in defining these algebras. As an application, we obtain a sufficient condition on the semi-simplicty of  these cyclotomic Hecke-Clifford superalgebras  via a dimension comparison. As a byproduct, both generic non-degenerate and degenerate cyclotomic Hecke-Clifford superalgebras are shown to be semisimple.

\end{abstract}
\maketitle

 \setcounter{tocdepth}{1}
\tableofcontents

\section{Introduction}

The representation theory of Hecke algebras associated with symmetric groups, along with their cyclotomic generalizations-known as cyclotomic Hecke algebras or Ariki-Koike algebras has been extensively studied in the literature (see the review papers \cite{Ar2, K3, Ma1} and references therein). In particular, Ariki and Koike provided in \cite{AK} a construction of the simple modules of generic cyclotomic Hecke algebras. Moreover, Ariki \cite{Ar1} established a necessary and sufficient condition for determining the semisimplicity for general cyclotomic Hecke algebras. {\color{black} In precise, let $\mathscr{H}_n(v,\underline{Q})$ be the cyclotomic Hecke algebra with Hecke parameter $v\neq1$ and cyclotomic parameter $\underline{Q}={\color{black}(Q_1,Q_2,\ldots,Q_l)}$. {\color{black} In \cite{Ar1}, }Ariki defined the following polynomial \begin{equation}
		P_{\mathscr{H}}(v,\underline{Q}):=\prod_{i=1}^{n}(1+v+\cdots+v^{i-1})(\prod_{\substack{1\leq i<j\leq l\\|d|<n}}(v^dQ_i-Q_j))
		\end{equation} and showed that $\mathscr{H}_n(v,\underline{Q})$  is semisimple if and only if $P_{\mathscr{H}}(v,\underline{Q})\neq 0$. This semisimple criterion enables the definition of a semisimple deformation for cyclotomic Hecke algebras that may not initially be semisimple. }In \cite{HM1, HM2}, several graded cellualr bases were constructed using semisimple deformation and seminormal bases \cite{Ma2} for cyclotomic Hecke algebras.  Recently, Evseev and Mathas \cite{EM} have used content system to extend these results to cyclotomic quiver Hecke algebras of types $A$ and $C$.

The Hecke-Clifford superalgebra \(\HC\), a quantum deformation of the Sergeev superalgebra \(\Cl_n \rtimes S_n\) (also referred to as the degenerate Hecke-Clifford superalgebra), was introduced by Olshanski \cite{Ol} to establish a super-analog of Schur-Jimbo duality involving the quantum enveloping algebra of the queer Lie superalgebra. {\color{black}The affine and cyclotomic analogue  of  Hecke-Clifford superalgebra \(\HC\) and Sergeev superalgebra \(\Cl_n \rtimes S_n\) were introduced in \cite{JN,N2}.  The representation theory of these algebras has been systematically studied, }revealing intriguing connections to Lie algebra representations \cite{BK1, BK2} (see also \cite{K1} for additional references). Recently progress on the Hecke-Clifford superalgebras and its various generalization including affine and cyclotomic analogues as well as analogues in other types rather than type A can be found in \cite{HKS, KL, W, WW1, WW2} and references therein. 

{\color{black}In \cite{JN}, the Hecke-Clifford superalgebra \(\HC\) with a generic quantum parameter was proved to be semisimple and a complete set of simple modules for \(\HC\) in terms of strict partitions and standard tableaux was constructed, generalizing \cite{N1, N2} for Sergeev superalgebra \(\Cl_n \rtimes S_n\). This framework connects closely with Schur’s work on the spin representations of symmetric groups \cite{Sch}.  In the cyclotomic situation, a construction of simple modules using induction and restriction functors has been obtained in [BK1, BK2]. It is an interesting problem to construct these simple modules via an explicit basis and actions in terms of multipartitions and standard tableaux analogous to the construction in \cite{Ar1} for cyclotomic Hecke algebras. Moreover, giving a necessary and sufficient condition for determining the semisimplicity for general non-degenerate and degenerate cyclotomic Hecke-Clifford superalgebras is still an open problem. These motivate the study of our work.
}

 Here is a quick summary of the main results of this paper. We first introduce the notion of separate parameters for cyclotomic Hecke-Clifford superalgebras and derive the equivalent description in terms of polynomials in the parameters. We then give an explicit construction of a class of non-isomorphic irreducible representations for cyclotomic Hecke-Clifford superalgebras with separate parameters. By a dimension comparison, we show that these irreducible representations exactly exhaust all non-isomorphic irreducible representations and moreover the cyclotomic Hecke-Clifford superalgebras with separate parameters are semsimple. 
 
Let us describe in some detail.  We first introduce three sets $ \mathscr{P}^{m}_{n},  \mathscr{P}^{\mathsf{s},m}_{n},  \mathscr{P}^{\mathsf{ss},m}_{n}$ of multipartitions $\undla=(\lambda^{(1)},\ldots,\lambda^{(m)})$, $\undla=(\lambda^{(0)}, \lambda^{(1)},\ldots,\lambda^{(m)})$ or $\undla=(\lambda^{(0_-)},\lambda^{(0_+)}, \lambda^{(1)},\ldots,\lambda^{(m)})$ which are mix of partitions $\lambda^{(1)},\ldots,\lambda^{(m)}$ and strict partitions $\lambda^{(0)}, \lambda^{(0_-)},\lambda^{(0_+)}$,  respectively. These three sets turn out to correspond to the three types of polynomials $f\in\{f^{(0)}_{\underline{Q}}(X_1), f^{(\mathsf{s})}_{\underline{Q}}(X_1), f^{(\mathsf{ss})}_{\underline{Q}}(X_1)\}$ in the definition of cyclotomic Hecke-Clifford superalgebras $\mHfcn$ with $\underline{Q}=(Q_1,\ldots,Q_m)$.  {\color{black}Then we introduce the separate condition for the parameters $q$ and $\underline{Q}$ in terms of the contents (also called residues) of the boxes in the multipartitions. We also introduce a polynomial $P_n^{(\bullet)}(q^2,\underline{Q})\in\{P_n^{(\mathsf{0})}(q^2,\underline{Q}), P_n^{(\mathsf{s})}(q^2,\underline{Q}), P_n^{(\mathsf{ss})}(q^2,\underline{Q})\}$ in $q^2$ and $\underline{Q}$, which can be regarded as the generalization of Poincar\'{e} polynomial for the Hecke algebras associated to symmetric groups. It turnes out that $P_n^{(\bullet)}(q^2,\underline{Q})\neq 0$ if and only if for any $\underline{\mu}\in\mathscr{P}^{\bullet,m}_{n+1}$, $(q,\undQ)$ is separate with respect to $\underline{\mu}$. We show that whenever $P_n^{(\bullet)}(q^2,\underline{Q})\neq 0$, }we can define an irreducible representation $\mathbb{D}(\undla)$ over the cyclotomic Hecke-Clifford superalgebra $\mHfcn$ for each multipartition $\undla\in \mathscr{P}^{m}_{n},  \mathscr{P}^{\mathsf{s},m}_{n},  \mathscr{P}^{\mathsf{ss},m}_{n}$, respectively. Our approach to obtaining the explicit construction of $\mathbb{D}(\undla)$ is inspired by the construction of the so-called completely splittable irreducible representations of degenerate affine Hecke-Clifford superalgebras established in \cite{Wa} (which generalizes \cite{K2, Ra, Ru}) by the {\color{black} second author}. More precisely, for each multipartition $\undla$, we {\color{black}define a space $\mathbb{D}(\undla)$ which is a direct sum of the irreducible $\mathcal{A}_n$-modules associated to content sequences corresponding to standard tableaux of type $\undla$, where $\mathcal{A}_n$ is the subalgebra of the affine Hecke-Clifford algebras generated by $X^{\pm1}_1,\ldots,X^{\pm1}_n$ and $C_1,\ldots,C_n$.  Moreover, the action of $X_1$ satisfies $f_{\underline{Q}}(X_1)=0$. }Then for separate parameters, we are able to define the action of $T_1,\ldots,T_{n-1}$ on $\mathbb{D}(\undla)$ and this makes $\mathbb{D}(\undla)$ admit a $\mHfcn$-module. 

 It is known that the Robinson–Schensted–Knuth correspondence for standard tableaux and its generalizations including for  standard tableaux for multipartitions and shifted standard tableaux gives rise to nice formulas (cf. \cite{DJM, Sa} ) involving the numbers of associated standard tableaux, see Lemma \ref{lem:formula}. This together with the dimension formula for the irreducible modules $\mathbb{D}(\undla)$ shows that the sum of the square of all dimensions of $\mathbb{D}(\undla)$ with proper powers of $2$ coincides with the dimension of the cyclotomic Hecke-Clifford superalgebra $\mHfcn$.  By Wedderburn Theorem for  associative superalgebras, we deduce that the cyclotomic Hecke-Clifford superalgebra $\mHfcn$ with separate parameters $q$ and $\underline{Q}$ is semisimple.  {\color{black}Our main result can be stated as follows, where the unexplained notation can be found in Section \ref{preliminary} and \ref{cyclotomic}.}
{ \color{black}\begin{thm}
 	Let  $q\neq \pm 1\in\mathbb{K}^*$ and $\undQ=(Q_1,Q_2,\ldots,Q_m)\in(\mathbb{K}^*)^m$.  Assume $f=f^{(\bullet)}_{\undQ}$ and   $P^{(\bullet)}_{n}(q^2,\undQ)\neq 0$, with $\bullet\in\{\mathtt{0},\mathtt{s},\mathtt{ss}\}$. Then  $\mHfcn$ is a (split) semisimple algebra and 
 	$$
 	\{\mathbb{D}(\undla)|~ \undla\in\mathscr{P}^{\bullet,m}_{n}\}$$ forms a complete set of pairwise non-isomorphic irreducible $\mHfcn$-module. Moreover,  $\mathbb{D}(\undla)$ is of type  $\texttt{M}$ if and only if $\sharp \mathcal{D}_{\undla}$  is even and is of type  $\texttt{Q}$ if and only if $\sharp \mathcal{D}_{\undla}$  is odd.  	
 \end{thm} 
}
{\color{black}As an application, we deduce that generic non-degenerate cyclotomic Hecke-Clifford superalgebras are semisimple. We remark that the above construction also works for degenerate cyclotomic Hecke-Clifford superalgebras with the notion of separate parameters modified accordingly.  We conjecture that condition $P^{(\bullet)}_{n}(q^2,\undQ)\neq 0$ is also necessary for the cyclotomic Hecke-Clifford superalgebras to be semisimple and we will work on this in a future project. 
}

Here is the layout of the paper. In Section \ref{preliminary}, we review some basics on superalgebras, multipartitions and standard tableaux  and set up various notations needed in the remainder of the paper. 
In Section \ref{cyclotomic}, we recall the notion of cyclotomic Hecke-Clifford superalgebras and introduce the separate parameters. 
In Section 4,  we construct an irreducible $\mHfcn$-module for each $\undla$ by assuming the parameters $q$ and $\underline{Q}$ are separate  and show that  in this situation these $\mathbb{D}(\undla)$ exhaust all  non-isomorphic irreducible $\mHfcn$-module  by a dimension comparison. Hence the superalgebra $\mHfcn$ with separate parameters is proved to be semisimple. 
{\color{black}In Section 5, we establish the analogue of the construction in Section 3 and Section 4 for degenerate cyclotomic Hecke-Clifford superalgebras. }

Throughout the paper, we shall assume $\mathbb{K}$ is an algebraically closed field of characteristic different from $2$.  Let $\mathbb{N}=\{1,2,\ldots\}$ be the set of positive integers and denote by $\mathbb{Z}_+$ the set of non-negative integers.

{\bf Acknowledgements.} The first author is grateful to Max Planck Institute for Mathematics in Bonn for its hospitality and financial support. The first author is also supported by the Natural Science Foundation of Beijing Municipality (No. 1232017). He thanks Shuo Li for many
useful discussions and feedback. The second author is supported by
NSFC 12122101 and NSFC 12071026. Both the authors are supported by the National Natural Science Foundation of China
(No. 12431002).

\section{Preliminary}\label{preliminary}
\subsection{Some basics about superalgebras} 
 We shall recall some basic notions of superalgebras, referring the
reader to~\cite[\S 2-b]{BK1}. {\color{black} By a superspace over $\mathbb{K}$, we mean a $\mathbb{Z}_2$-graded vector space.} Let us denote by
$\bar{v}\in\mathbb{Z}_2$ the parity of a homogeneous vector $v$ of a
vector superspace. By a superalgebra, we mean a
$\mathbb{Z}_2$-graded associative algebra. Let $\mathcal{A}$ be a
superalgebra. A $\mathcal{A}$-module means a $\mathbb{Z}_2$-graded
left $\mathcal{A}$-module. A homomorphism $f:V\rightarrow W$ of
$\mathcal{A}$-modules $V$ and $W$ means a linear map such that $
f(av)=(-1)^{\bar{f}\bar{a}}af(v).$  Note that this and other such
expressions only make sense for homogeneous $a, f$ and the meaning
for arbitrary elements is to be obtained by extending linearly from
the homogeneous case.  Let $V$ be a finite dimensional
$\mathcal{A}$-module. Let $\Pi
 V$ be the same underlying vector space but with the opposite
 $\mathbb{Z}_2$-grading. The new action of $a\in\mathcal{A}$ on $v\in\Pi
 V$ is defined in terms of the old action by $a\cdot
 v:=(-1)^{\bar{a}}av$. Note that the identity map on $V$ defines
 an isomorphism from $V$ to $\Pi V$. By forgeting the grading we may consider any superalgebra $\mathcal{A}$ as the usual algebra which will be denoted by $|\mathcal{A}|$. Similarly, any $\mathcal{A}$-supermodule $V$ can be considered as a usual $|\mathcal{A}|$-module denoted by $|V|$. A superalgebra analog of Schur's Lemma (cf. \cite{K2}) states that the endomorphism
 algebra $\text{End}_{\mathcal{A}}(M)$ of a finite dimensional irreducible module $\mathcal{A}$-module $V$  is either one dimensional or two dimensional. In the
 former case, we call the module $M$ of {\em type }\texttt{M} while in
 the latter case the module $V$ is called of {\em type }\texttt{Q}.  


\begin{lem}\cite[Lemma 12.2.1, Corollary 12.2.10]{K1}\label{lem:type MQ}
	Suppose $V$ is an irreducible $\mathcal{A}$-module. If $V$ is of type $\texttt{M}$, then by forgetting the grading, $|V|$ is an irreducible $|\mathcal{A}|$-module. If $V$ is of type $\texttt{Q}$, then by forgetting the grading, $|V|$  is isomorphic to a direct sum of two non-isomorphic irreducible $|\mathcal{A}|$-modules. That is, there exist two non-isomorphic irreducible $|\mathcal{A}|$-modules $V^+,V^-$ such that $|V|\cong V^+\oplus V^-$ as $|\mathcal{A}|$-modules. Moreover if $V_1,\cdots,V_m$ (resp. $V_{m+1},\cdots,V_n$) are pairwise non-isomorphic irreducible $\mathcal{A}$-modules of type \texttt{M} (resp. \texttt{Q}), then $$\{|V_1|, \cdots,|V_m|, V_{m+1}^\pm, \cdots, V_{n}^\pm\}$$ is a set of pairwise non-isomorphic $|\mathcal{A}|$-modules. 
\end{lem} 

Denote by $\mathcal{J}(\mathcal{A})$  the usual (non-super) Jacobson radical of $\mathcal{A}$,  that is $\mathcal{J}(\mathcal{A})=\mathcal{J}(|\mathcal{A}|)$. We call $\mathcal{A}$ is semisimple if $\mathcal{J}(\mathcal{A})=0$. By Lemma \ref{lem:type MQ}, for any irreducible $\mathcal{A}$-module $V$, we have 
\begin{equation}\label{eq:JA-nill-V}
	\mathcal{J}(\mathcal{A}) V=0. 
	\end{equation}

\begin{lem}\cite[Lemma 12.2.7]{K1}\label{lem:semisimple1}
Let $\mathcal{A}$ be a finite dimensional superalgebra. Then $\mathcal{A}/\mathcal{J}(\mathcal{A})$ is semisimple.
\end{lem}
%
%
%

\begin{cor}\label{cor:dim-compare}
Let $\mathcal{A}$ be a finite dimensional superalgebra. Suppose $\{V_1, V_2,\ldots, V_s\}$ is a class of non-isomorphic irreducible $\mathcal{A}$-modules of type \texttt{M} and $\{U_1,U_2,\ldots, U_t\}$ is a class of non-isomorphic irreducible $\mathcal{A}$-modules of type \texttt{Q}. If
\begin{equation}\label{eq:dimequal}
\dim \mathcal{A}=\sum_{i=1}^s(\dim V_i)^2+\sum_{j=1}^t\frac{(\dim U_j)^2}{2}, 
\end{equation}
then $\mathcal{A}$ is semisimple. 
\end{cor}
\begin{proof}
By \eqref{eq:JA-nill-V}, each $V_i$ and $U_j$ are annihilated by $\mathcal{J}(\mathcal{A})=\mathcal{J}(|\mathcal{A}|)$ and hence all $V_i$ and $U_j$ admit $|A|/\mathcal{J}(|\mathcal{A}|)$-modules for $1\leq i\leq s, 1\leq j\leq t$. Moreover by applying Wedderburn Theorem to the usual algebra $|\mathcal{A}|/\mathcal{J}(|\mathcal{A}|)$ which is semisimple according to Lemma \ref{lem:semisimple1}   and then by Lemma \ref{lem:type MQ}, we have 
$$
\dim |\mathcal{A}|/\mathcal{J}(|\mathcal{A}|)\geq \sum_{i=1}^s(\dim V_i)^2+\sum_{j=1}^t((\dim U^+_j)^2+(\dim U_j^-)^2)=\sum_{i=1}^s(\dim V_i)^2+\sum_{j=1}^t\frac{(\dim U_j)^2}{2}.
$$
This together with the assumption \eqref{eq:dimequal} leads to $\mathcal{J}(\mathcal{A})=0$ as $\dim \mathcal{A}=\dim |\mathcal{A}|$ and we obtain that $\mathcal{A}$ is semisimple {\color{black} by Lemma \ref{lem:semisimple1}}.

\end{proof}

Given two superalgebras $\mathcal{A}$ and $\mathcal{B}$, we view
the tensor product of superspaces $\mathcal{A}\otimes\mathcal{B}$
as a superalgebra with multiplication defined by
$$
(a\otimes b)(a'\otimes b')=(-1)^{\bar{b}\bar{a'}}(aa')\otimes (bb')
\qquad (a,a'\in\mathcal{A}, b,b'\in\mathcal{B}).
$$
Suppose $V$ is an $\mathcal{A}$-module and $W$ is a
$\mathcal{B}$-module. Then $V\otimes W$ affords $A\otimes B$-module
denoted by $V\boxtimes W$ via
$$
(a\otimes b)(v\otimes w)=(-1)^{\bar{b}\bar{v}}av\otimes bw,~a\in A,
b\in B, v\in V, w\in W.
$$
If $V$ is an irreducible $\mathcal{A}$-module and $W$ is an
irreducible $\mathcal{B}$-module, $V\boxtimes W$ may not be
irreducible. Indeed, we have the following standard lemma (cf.
\cite[Lemma 12.2.13]{K1}).
\begin{lem}\label{tensorsmod}
Let $V$ be an irreducible $\mathcal{A}$-module and $W$ be an
irreducible $\mathcal{B}$-module.
\begin{enumerate}
\item If both $V$ and $W$ are of type $\texttt{M}$, then
$V\boxtimes W$ is an irreducible
$\mathcal{A}\otimes\mathcal{B}$-module of type $\texttt{M}$.

\item If one of $V$ or $W$ is of type $\texttt{M}$ and the other one
is of type $\texttt{Q}$, then $V\boxtimes W$ is an irreducible
$\mathcal{A}\otimes\mathcal{B}$-module of type $\texttt{Q}$.

\item If both $V$ and $W$ are of type $\texttt{Q}$, then
$V\boxtimes W\cong X\oplus \Pi X$ for a type $\texttt{M}$
irreducible $\mathcal{A}\otimes\mathcal{B}$-module $X$.
\end{enumerate}
Moreover, all irreducible $\mathcal{A}\otimes\mathcal{B}$-modules
arise as constituents of $V\boxtimes W$ for some choice of
irreducibles $V,W$.
\end{lem}
If $V$ is an irreducible $\mathcal{A}$-module and $W$ is an
irreducible $\mathcal{B}$-module, denote by $V\circledast W$ an
irreducible component of $V\boxtimes W$. Thus,
$$
V\boxtimes W=\left\{
\begin{array}{ll}
V\circledast W\oplus \Pi (V\circledast W), & \text{ if both } V \text{ and } W
 \text{ are of type }\texttt{Q}, \\
V\circledast W, &\text{ otherwise }.
\end{array}
\right.
$$

\subsection{Some combinatorics}
For $n\in \N$, let $\mathscr{P}_n$ be the set of partitions of {\color{black}$n$} and denote by $\ell(\mu)$ the number of nonzero parts in the partition $\mu$ for each $\mu\in\mathscr{P}_n$. Let $\mathscr{P}^m_n$ be the set of all $m$-multipartitions of $n$ for $m\geq 0$, where we use convention that $\mathscr{P}^0_n=\emptyset$. Let $\mathscr{P}^\mathsf{s}_n$ be the set of strict partitions of $n$. Then for $m\geq 0$, set
 $$
 \mathscr{P}^{\mathsf{s},m}_{n}:=
\cup_{a=0}^{n}( \mathscr{P}^{\mathsf{s}}_a\times \mathscr{P}^{m}_{n-a}),\qquad \mathscr{P}^{\mathsf{ss}, m}_{n}:=
\cup_{a+b+c=n}(\mathscr{P}^{\mathsf{s}}_a \times~ \mathscr{P}^{\mathsf{s}}_b\times \mathscr{P}^{m}_{c})$${\bf We will formally write  $\mathscr{P}^{\mathsf{0},m}_{n}=\mathscr{P}^m_n$.  In convention,   for any $\undla\in  \mathscr{P}^{\mathsf{0},m}_{n}$, we write  $\undla=(\lambda^{(1)},\cdots,\lambda^{(m)})$ while for any $\undla\in  \mathscr{P}^{\mathsf{s},m}_{n}$, we write  $\undla=(\lambda^{(0)},\lambda^{(1)},\cdots,\lambda^{(m)})$, i.e. we shall put the strict partition in the $0$-th component. {\color{black} In addition}, for any $\undla\in  \mathscr{P}^{\mathsf{ss},m}_{n}$, we write  $\undla=(\lambda^{(0_-)},\lambda^{(0_+)},\lambda^{(1)},\cdots,\lambda^{(m)})$, i.e. we shall put two strict partitions in the $0_-$-th component and the $0_+$-th component.  }{\color{black} We will see the justification of the choice of the notations $0_+,0_-$ later on in the definition of cyclotomic Hecke-Clifford superalgebras.}


We will also identify the (strict) partition with the corresponding (shifted) young diagram.  For any  $\undla\in\mathscr{P}^{\bullet,m}_{n}$ with $\bullet\in\{\mathsf{0},\mathsf{s},\mathsf{ss}\}$ and $m\in \N$, the box in the $l$-th component with row $i$, colum $j$ will be denoted by $(i,j,l)$  with $l\in\{1,2,\ldots,m\},$ or $l\in\{0,1,2,\ldots,m\}$ or $l\in\{0_-,0_+,1,2,\ldots,m\}$ in the case $\bullet=\mathsf{0},\mathsf{s},\mathsf{ss},$ respectively. We also use the notation $\alpha=(i,j,l)\in \undla$ if the diagram of $\undla$ has a box $\alpha$ on the $l$-th component of row $i$ and colum $j$. We use $\Std(\undla)$ to denote the set of standard tableaux of shape $\undla$. One can also regard each $\mathfrak{t}\in\Std(\undla)$ as a bijection $\mathfrak{t}:\undla\rightarrow \{1,2,\ldots, n\}$ satisfying $\mathfrak{t}((i,j,l))=k$ if the box occupied by $k$ is located in the $i$th row, $j$th column in the $l$-th component $\lambda^{(l)}$. We use $\mathfrak{t}^{\undla}$ to denote the standard tableaux obtained by inserting the symbols $1,2,\ldots,n$ consecutively by rows from the first component of $\undla$. 

\begin{defn} Let $\undla\in\mathscr{P}^{\bullet,m}_{n}$ with $\bullet\in\{\mathsf{0},\mathsf{s},\mathsf{ss}\}$.  We define$$
	\mathcal{D}_{\undla}:=\begin{cases} \emptyset,&\text{if $\undla\in\mathscr{P}^{\mathsf{0},m}_n$,}\\
		\{(a,a,0)|(a,a,0)\in \undla,\,a\in\N\}, &\text{if $\undla\in\mathscr{P}^{\mathsf{s},m}_{n}$,}\\
	\big\{(a,a,l)|(a,a,l)\in \undla,\,a\in\N, l\in\{0_-,0_+\}\big\}, &\text{if $\undla\in\mathscr{P}^{\mathsf{ss}, m}_{n}.$}
		\end{cases}
	$$ For any $\mathfrak{t}\in\Std(\undla), $ we define $$\mathcal{D}_{\mathfrak{t}}:=\{\mathfrak{t}((a,a,l))|(a,a,l)\in\mathcal{D}_{\undla}\}.
	$$
\end{defn}

\begin{example} Let $\undla=(\lambda^{(0)},\lambda^{(1)})\in \mathscr{P}^{\mathsf{s},1}_{5}$, where via the identification with strict Young diagrams and Young diagrams: 
 $$
	\lambda^{(0)}=\young(\,\, ,:\,),\qquad \lambda^{(1)}=\young(\,,\,).
	$$
Then 
$$
\mathfrak{t}^{\undla}=\Biggl(\young(12,:3),\quad \young(4,5)\Biggr).
$$ 
and  an example of standard tableau is as follows: 
$$
\mathfrak{t}=\Biggl(\young(13,:5),\quad \young(2,4)\Biggr)\in \Std(\undla). 
$$ We have $$\mathcal{D}_{\undla}=\{(1,1,0),(2,2,0)\},\qquad \mathcal{D}_{\mathfrak{t}}=\{1,5\}.
$$
\end{example}

Let $\mathfrak{S}_n$ be the symmetric group on ${1,2,\ldots,n}$ with basic transpositions $s_1,s_2,\ldots, s_{n-1}$.  Clearly $\mathfrak{S}_n$ acts on the set of tableaux of shape $\underline{\lambda}$. 
\begin{defn}Let  $\undla\in\mathscr{P}^{\bullet,m}_{n}$ with $\bullet\in\{\mathsf{0},\mathsf{s},\mathsf{ss}\}$.  For any standard tableaux $\mathfrak{t}\in \Std(\undla)$, if $s_l \mathfrak{t}$ is still standard, {\color{black} the simple transposition $s_l$ is said to be admissible} with respect to $\mathfrak{t}$. We set $$P(\undla):=\biggl\{\tau=s_{k_t}\ldots s_{k_1}\biggm|~\begin{matrix}&s_{k_l} \,\text{is admissible with respect to }\\
		& s_{k_{l-1}}\ldots s_{k_1}\mathfrak{t}^{\undla}, \,\text{for $l=1,2,\cdots,t$}
		\end{matrix}\biggr\}.
$$ 
\end{defn}
 The following results should be known, however we did not find the detail proof in literature and hence we include one in the following. 
\begin{lem}\label{lem:admissible}
  Let  $\undla\in\mathscr{P}^{\bullet,m}_{n}$ with $\bullet\in\{\mathsf{0},\mathsf{s},\mathsf{ss}\}$.  
  \begin{enumerate} 
  	\item Let $\mathfrak{s}\in \Std(\undla)$ and {\color{black}$i\in[1, n-2]$}. Then $s_i\mathfrak{s},s_{i+1}s_i\mathfrak{s},s_is_{i+1}s_i\mathfrak{s}\in \Std(\undla)$ if and only if $s_{i+1}\mathfrak{s},s_{i}s_{i+1}\mathfrak{s},s_{i+1}s_{i}s_{i+1}\mathfrak{s}\in \Std(\undla)$.
  	\item Let $\mathfrak{s}\in \Std(\undla)$ and $i,j\in[1, n-1]$ such that $|i-j|>1$. Then $s_i\mathfrak{s},s_{j}s_i\mathfrak{s}\in \Std(\undla)$ if and only if $s_{j}\mathfrak{s},s_{i}s_{j}\mathfrak{s}\in \Std(\undla)$.
  	\item Let $\mathfrak{s}, \mathfrak{t}\in \Std(\undla),\,\tau\in\mathfrak{S}_n$ such that $\tau\mathfrak{s}=\mathfrak{t}$. Let $\tau=s_{i_k}\cdots s_{i_1}$ be any reduced expression of $\tau$. Then $s_{i_l}$ is admissible with respect to $s_{i_{l-1}}\cdots s_{i_1}\mathfrak{s}$ for $l=1,2,\cdots,k$.
  	\end{enumerate}
\end{lem}

 \begin{proof}
 	Observe that for any $\mathfrak{t}\in \Std(\undla)$ and $1\leq k\leq n-1$, $s_k\mathfrak{t}$ is standard if and only if $k,k+1$ are not adjacent in $\mathfrak{t}$.
	
	(1)  Let $\mathfrak{s}\in \Std(\undla)$ and {\color{black}$i\in[1,n-2]$}. The observation implies that  $s_i\mathfrak{s},s_{i+1}s_i\mathfrak{s},s_is_{i+1}s_i\mathfrak{s}$ are all standard  if and only if $j,k$ are not adjacent in $\mathfrak{s}$ for any $j\neq k\in\{i,i+1,i+2\}$. Similarly, $s_{i+1}\mathfrak{s},s_{i}s_{i+1}\mathfrak{s},s_{i+1}s_{i}s_{i+1}\mathfrak{s}\in \Std(\undla)$ 
if and only if $j,k$ are not adjacent in $\mathfrak{s}$ for any $j\neq k\in\{i,i+1,i+2\}$. Hence (1) holds.

 (2) Let $\mathfrak{s}\in \Std(\undla)$ and {\color{black}$i,j\in[1,n-1]$} such that $|i-j|>1$. Then by the above observation again, we obtain that $s_i\mathfrak{s},s_{j}s_i\mathfrak{s}\in \Std(\undla)$(or $s_{j}\mathfrak{s},s_{i}s_{j}\mathfrak{s}\in \Std(\undla)$) if and only if $i,i+1$ are not adjacent in $\mathfrak{s}$ and $j,j+1$ are not adjacent  in $\mathfrak{s}$.  Hence (2) holds. 
 	
	(3) For $\mathfrak{s}, \mathfrak{t}\in \Std(\undla),\,\tau\in\mathfrak{S}_n$ such that $\tau\mathfrak{s}=\mathfrak{t}$, we first claim: \begin{align}&\text{$\exists$ a reduced expression of $\tau=s_{i_k}\cdots s_{i_1}$, such that $s_{i_l}$ is}\nonumber\\
 			&\text{ admissible with respect to $s_{i_{l-1}}\cdots s_{i_1}\mathfrak{s}$, for $l=1,2,\cdots,k$. }\label{tiny claim}
 			\end{align}
We define $O(\mathfrak{s},\mathfrak{t})$ to be the maximal $m\leq n$ such that {\color{black}$1,2,\cdots,m-1$} in $\mathfrak{s}$ are at the same position as in $\mathfrak{t}$. Clearly $O(\mathfrak{s}, \mathfrak{t})=n$ if and only if $\mathfrak{s}=\mathfrak{t}$. Then we use induction downward on $O(\mathfrak{s},\mathfrak{t})$ to prove \eqref{tiny claim}. Actually, if $m=O(\mathfrak{s}, \mathfrak{t})<n$ and $\mathfrak{s}(i,j,l)=m$, then $\mathfrak{t}(i,j,l)=m'>m$. By our definition of $O(\mathfrak{s}, \mathfrak{t})$,  it is clear that for any $m\leq u<m'$, $u$ and $m'$ are not adjacent in $\mathfrak{t}$. This is equivalent to say each $s_u$ is admissible with respect to $s_{u}\cdots s_{m'-1}\mathfrak{t}$ for $m\leq u<m'$ and hence $s_{m}\cdots s_{m'-1}\mathfrak{t}\in\Std(\undla)$.  Note that $O(\mathfrak{s}, s_{m}\cdots s_{m'-1}\mathfrak{t})>m$.  By induction, for $\tau'\in\Sym_{\{m,m+1,\cdots,n\}}$ such that $\tau'\mathfrak{s}=s_{m}\cdots s_{m'-1}\mathfrak{t}$, {\color{black} the claim}\eqref{tiny claim} holds. That is, {\color{black} there exists} a reduced expression of $\tau=s_{i'_{k'}}\cdots s_{i'_1}$, such that $s_{i'_l}$ is admissible with respect to $s_{i'_{l-1}}\cdots s_{i'_1}'\mathfrak{s}$, for $l=1,2,\cdots,k'$. Combing with that $s_{m'-1}\cdots s_{m+1}s_{m}$ is a minimal left representative of $\Sym_{\{m,m+1,\cdots,n\}}$ in $\Sym_n$, we deduce that \eqref{tiny claim} holds for $\mathfrak{s},\,\mathfrak{t}$ and $\tau$. Now (3) follows from \eqref{tiny claim}, part (1), (2) of the Lemma  and Matsumoto's Lemma.
  \end{proof}
 
 \begin{cor}\label{bij}
 	There is a bijection $\psi:\, P(\undla)\rightarrow \Std(\undla)$.
 	\end{cor}
 	
 	\begin{proof}
 		  Set $\psi(\tau):=\tau\mathfrak{t}^{\undla}$.  Clearly, $\psi$ is a well-defined injective map. By Lemma \ref{lem:admissible} (3), $\psi$ is surjective.
 		\end{proof}

\section{Cyclotomic Hecke-Clifford superalgebras and separate parameters}\label{cyclotomic}
\subsection{Affine Hecke-Clifford algebra $\mHcn$}
Let $q\neq \pm 1$ be an invertible element in $\mathbb{K}$ and  set 
$$\epsilon=q-q^{-1}.$$
 It follows from  \cite{JN} that the non-degenerate affine Hecke-Clifford algebra $\mHcn$ is
the superalgebra over $\mathbb{K}$ generated by even generators
$T_1,\ldots,T_{n-1},X^{\pm 1}_1,\ldots,X^{\pm 1}_n$ and odd generators
$C_1,\ldots,C_n$ subject to the following relations
\begin{align}
T_i^2=\epsilon T_i+1,\quad T_iT_j =T_jT_i, &\quad
T_iT_{i+1}T_i=T_{i+1}T_iT_{i+1}, \quad|i-j|>1,\label{Braid}\\
X_iX_j&=X_jX_i, X_iX^{-1}_i=X^{-1}_iX_i=1 \quad 1\leq i,j\leq n, \label{Poly}\\
C_i^2=1,C_iC_j&=-C_jC_i, \quad 1\leq i\neq j\leq n, \label{Clifford}\\
T_iX_i&=X_{i+1}T_i-\epsilon(X_{i+1}+C_iC_{i+1}X_i),\label{PX1}\\
T_iX_{i+1}&=X_iT_i+\epsilon(1+C_iC_{i+1})X_{k+1},\label{PX2}\\
T_iX_j&=X_jT_i, \quad j\neq i, i+1, \label{PX3}\\
T_iC_i=C_{i+1}T_i, T_iC_{i+1}&=C_iT_i-\epsilon(C_i-C_{i+1}),T_iC_j=C_jT_i,\quad j\neq i, i+1, \label{PC}\\
X_iC_i=C_iX^{-1}_i, X_iC_j&=C_jX_i,\quad 1\leq i\neq j\leq n.
\label{XC}
\end{align}

For each permutation $w\in\mathfrak{S}_n$ with an reduced expression $w=s_{i_1}s_{i_2}\cdots s_{i_r}$ for some $1\leq i_1,\ldots,i_{r}\leq n-1$ with $r\geq 0$, there exits a element $T_w:=T_{i_1}\cdots T_{i_r}$ and it is independent of the choice of the reduced expression of $w$ due to the Braid relation \eqref{Braid}. For $\alpha=(\alpha_1,\ldots,\alpha_n)\in\mathbb{Z}^n$ and
$\beta=(\beta_1,\ldots,\beta_n)\in\mathbb{Z}_2^n$, set
$X^{\alpha}=X_1^{\alpha_1}\cdots X_n^{\alpha}$ and
$C^{\beta}=C_1^{\beta_1}\cdots C_n^{\beta_n}$. Then we have the
following.
\begin{lem}\cite[Theorem 2.3]{BK1}\label{lem:PBWNon-dege}
The set $\{X^{\alpha}C^{\beta}T_w~|~ \alpha\in\mathbb{Z}^n,
\beta\in\mathbb{Z}_2^n, w\in {\mathfrak{S}_n}\}$ forms a basis of $\mHcn$.
\end{lem}

{\color{black} The next lemma was first established in \cite[Proposition 3.2]{JN}
(for the case $\mathbb{K}=\mathbb{C}$).
\begin{lem}\cite[Theorem 2.2]{BK1}\label{lem:affine-center}
The (super)center of $\mHcn$ consists of all symmetric polynomials in
$X_1+X_1^{-1},X_2+X_2^{-1},\ldots,X_n+X_n^{-1}$. 
\end{lem}}

 Let $\mathcal{A}_n$ be the subalgebra generated by even generators $X^{\pm}_1,\ldots,X^{\pm 1}_n$ and odd generators $C_1,\ldots,C_n$. By Lemma~\ref{lem:PBWNon-dege}, $\mathcal{A}_n$ actually can be identified with the superalgebra generated by even generators $X^{\pm 1}_1,\ldots,X^{\pm 1}_n$ and odd generators $C_1,\ldots,C_n$ subject to relations \eqref{Poly}, \eqref{Clifford}, \eqref{XC}.

{\color{black}Recall that $\mathbb{K}$ is an algebraically closed field of characteristic different from $2$. For any $a\in \mathbb{K}$, we fix a solution of the equation $x^2=a$ and denote it by $\sqrt{a}$.} For any  $x \in \mathbb{K}^*$, we define \begin{align}\label{substitution0}
\mathtt{q}(x):=2\frac{qx+(qx)^{-1}}{q+q^{-1}}, \quad b_{\pm}(x):=\frac{\mathtt{q}(x)}{2}\pm \sqrt{\frac{\mathtt{q}(x)^2}{4}-1}.
\end{align} We remark that $\mathtt{q}(q^{2i})$ is the definition of $q(i)$ in \cite[(4.5)]{BK1}. Clearly, $\mathtt{b}_{\pm}(x)$ are exactly two solutions satisfying the equation $z+z^{-1}=\mathtt{q}(x)$ and moreover 
\begin{equation}\label{bpm}
\mathtt{b}_-(x)=\mathtt{b}_+(x)^{-1}. 
\end{equation}
One can easily check that $\mathtt{q}(x)=\mathtt{q}(y)$ for $x,y\in\mathbb{K}^*$ if and only if either $x=y$ or $xy=q^{-2}$. Or equivalently 
$\{\mathtt{b}_+(x),\mathtt{b}_-(x) \}\cap\{\mathtt{b}_+(y),\mathtt{b}_-(y) \}\neq\emptyset$ if and only if either $x=y$ or $xy=q^{-2}$. 

We define an equivalent relation $\sim$ on $\mathbb{K}^*$ by $x\sim y$ if and only if $x=y$ or $xy=q^{-2}$. Let $\mathcal{K}$ be the subset of $\mathbb{K}^*$  which contains exactly one representative element of $\sim$ and contains $\pm 1$ {\color{black}(hence $\pm q^{-2}$ are excluded).} 
Clearly, for $\iota_1\neq \iota_2\in \mathcal{K}$, we have $q(\iota_1)\neq q(\iota_2)$ and hence $b_{\pm}(\iota_1) \neq b_{\pm}(\iota_2) $. 
Moreover, 
\begin{equation}\label{Kinvt}
\{\mathtt{b}_\pm(\iota)|\iota\in\mathcal{K}\}=\mathbb{K}^*
\end{equation}
{\color{black} 
For each $x\in \mathbb{K}^*$, let $\mathbb{L}(x)$ be the $2$-dimensional
$\mathcal{A}_1$-module $\mathbb{L}(x)=\mathbb{K}v_0\oplus \mathbb{K}v_1$  with 
$$
X^{\pm 1}_1 v_0=\mathtt{b}_{\pm}(x)v_0,\quad X^{\mp 1}_1 v_1=\mathtt{b}_{\mp}(x)v_1, \quad
C_1v_0=v_1,\quad C_1v_1=v_0.
$$
Clearly $\mathbb{L}(x)\cong \mathbb{L}(y)$ if and only if $x=y$ or $xy=q^{-2}$. That is, $\mathbb{L}(x)\cong \mathbb{L}(y)$ if and only if $x\sim y$. Therefore for each $\iota\in \mathcal{K}$, there exits a  $2$-dimensional 
$\mathcal{A}_1$-module denoted by $\mathbb{L}(\iota)=\mathbb{K}v_0\oplus \mathbb{K}v_1$ with 
}
$$
X^{\pm 1}_1 v_0=\mathtt{b}_{\pm}(\iota)v_0,\quad X^{\mp 1}_1 v_1=\mathtt{b}_{\mp}(\iota)v_1, \quad
C_1v_0=v_1,\quad C_1v_1=v_0. 
$$

\begin{lem} 
The $\mathcal{A}_1$-module $\mathbb{L}(\iota)$  is irreducible of type $\texttt{M}$ if $\iota^2\neq 1 $,
and irreducible of type $\texttt{Q}$ if $\iota^2 = 1$. Moreover, $\{\mathbb{L}(\iota)|\iota\in\mathcal{K}\}$ is a complete set of pairwise non-isomorphic finite dimensional irreducible $\mathcal{A}_1$-module. 

\end{lem}
\begin{proof}
One can easily check the first statement holds since  $\mathtt{b}_+(\iota)=\mathtt{b}_-(\iota)=1$ in the case $\iota^2=1$.
Observe that for each $z\in \mathbb{K}^*$, there exists a $2$-dimensional
$\mathcal{A}_1$-module $U(z)=\mathbb{K}u_0\oplus \mathbb{K}u_1$  with 

$$
X^{\pm 1}_1 u_0=z^{\pm 1}u_0,\quad X^{\mp 1}_1 u_1=z^{\mp 1}u_1, \quad
C_1u_0=u_1,\quad C_1u_1=u_0.
$$
Moreover $U(z)\cong U(z')$ if and only if $z=z'$ or $z'=z^{-1}$. Hence $\mathbb{L}(\iota)=U(\mathtt{b}_+(\iota))\cong U(\mathtt{b}_-(\iota))$. 
Meanwhile given a finite dimensional irreducible $\mathcal{A}_1$-module $U$, there exists an eigenvector $u_0$ of $X_1$ on the action of the even space $U_{\bar{0}}$ with eigenvalue $z$ for some $z\in\mathbb{K}^*$. 
Then it is straightforward to check that $U\cong U(z)$.  Putting together, we obtain that $\{\mathbb{L}(\iota)|\iota\in\mathcal{K}\}$ is a complete set of pairwise non-isomorphic finite dimensional irreducible $\mathcal{A}_1$-module. 
\end{proof}
Clearly we have $$\mathcal{A}_n\cong \mathcal{A}_1\otimes\cdots\otimes \mathcal{A}_1.$$
{\color{black} For each $\underline{a}=(a_1,a_2,\ldots,a_n)\in(\mathbb{K}^*)^n$, set 
\begin{equation}\label{L-under-a}
\mathbb{L}(\underline{a})=\mathbb{L}(a_1)\circledast\mathbb{L}(a_2)\circledast\cdots\circledast\mathbb{L}(a_n), 
\end{equation}
 then $\mathbb{L}(\underline{a})\cong \mathbb{L}(\underline{b})$ if and only if $a_i\sim b_i$ for $1\leq i\leq n$. }
Then by Lemma \ref{tensorsmod}, we have the following result which can be view as a generalization of \cite[Lemma 4.8]{BK1}:

\begin{cor} \label{lem:irrepAn}
The $\mathcal{A}_n$-modules
$$
\{\mathbb{L}(\underline{\iota})=\mathbb{L}(\iota_1)\circledast
\mathbb{L}(\iota_2)\circledast\cdots\circledast
\mathbb{L}(\iota_n)|~\underline{\iota}=(\iota_1,\ldots,\iota_n)\in\mathcal{K}^n\}
$$
forms a complete set of pairwise non-isomorphic finite dimensional irreducible
$\mathcal{A}_n$-module. 
Moreover, denote by $\Gamma_0$ the number of $1\leq j\leq n$ with
$\iota_j^2=1$. Then $\mathbb{L}(\underline{\iota})$ is of type $\texttt{M}$ if
$\Gamma_0$ is even and type $\texttt{Q}$ if $\Gamma_0$ is odd.
Furthermore,
$\text{dim}~\mathbb{L}(\underline{\iota})=2^{n-\lfloor\frac{\Gamma_0}{2}\rfloor}$,
where $\lfloor\frac{\Gamma_0}{2}\rfloor$ denotes the greatest
integer less than or equal to $\frac{\Gamma_0}{2}$ .
\end{cor}

\begin{rem}\label{rem:Ltau1}
Following \cite[Remark 2.5]{Wa}, we observe that each permutation $\tau\in {\mathfrak{S}_n}$ defines a superalgebra
isomorphism $\tau:\mathcal{A}_n \rightarrow \mathcal{A}_n$ by mapping $X^{\pm 1} _k$ to
$X^{\pm 1}_{\tau(k)}$ and $C_k$ to  $C_{\tau(k)}$, for $1\leq k\leq n$. For
$\underline{\iota}\in\mathcal{K}^n$, the twist of the action of
$\mathcal{A}_n$ on $\mathbb{L}(\underline{\iota})$ with
$\tau^{-1}$ leads to a new $\mathcal{A}_n$-module denoted by
$\mathbb{L}(\underline{\iota})^{\tau}$ with
$$
\mathbb{L}(\underline{\iota})^{\tau}=\{z^{\tau}~|~z\in \mathbb{L}(\underline{\iota})\} ,\quad
fz^{\tau}=(\tau^{-1}(f)z)^{\tau}, \text{ for any }f\in
\mathcal{A}_n, z\in \mathbb{L}(\underline{\iota}).
$$
So in particular we have 
\begin{equation}\label{X-z-tau}
(X^{\pm 1} _kz)^{\tau}=X^{\pm 1}_{\tau(k)}z^{\tau}, 
(C_kz)^{\tau}=C_{\tau(k)}z^{\tau}
\end{equation}
 for each $1\leq k\leq n$. It is easy to see that $\mathbb{L}(\underline{\iota})^{\tau}\cong \mathbb{L}(\tau\cdot \underline{\iota})$, where
$\tau\cdot \underline{\iota}=(\iota_{\tau^{-1}(1)},\ldots,\iota_{\tau^{-1}(n)})$
for $\underline{\iota}=(\iota_1,\ldots,\iota_n)\in\mathcal{K}^n$ and $\tau\in {\mathfrak{S}_n}$. Moreover it is straightforward to show that the following holds 
\begin{equation}\label{L-tau-sigma}
((\mathbb{L}(\underline{\iota}))^\tau)^\sigma\cong \mathbb{L}(\underline{\iota})^{\sigma\tau}. 
\end{equation}

\end{rem}

\subsection{Intertwining elements for $\mHcn$}
Given $1\leq i<n$, we define the intertwining element $\widetilde{\Phi}_i$ in $\mHcn$ as follows: 
\begin{equation} \label{eq:zi}
\mathsf{z}_i:= (X_i+X^{-1}_i)-(X_{i+1}+X^{-1}_{i+1})= X^{-1}_i(X_iX_{i+1}-1)(X_iX^{-1}_{i+1}-1),
\end{equation}
\begin{equation}\label{intertwinNon-dege}
\widetilde{\Phi}_i:=\mathsf{z}^2_i T_i+\epsilon\frac{\mathsf{z}^2_i}{X_i X^{-1}_{i+1}-1}-\epsilon\frac{\mathsf{z}^2_i}{X_i X_{i+1}-1}C_i C_{i+1}.
\end{equation} These elements satisfy the following properties (cf. \cite[(3.7),Proposition 3.1]{JN} and \cite[(4.11)-(4.15)]{BK1}) 
\begin{align}
\widetilde{\Phi}^2_i&=\mathsf{z}^2_i\bigl(\mathsf{z}^2_i-\epsilon^2 (X^{-1}_i X^{-1}_{i+1}(X_i X_{i+1}-1)^2-X^{-1}_iX_{i+1}(X_i X^{-1}_{i+1}-1)^2)\bigr)\label{Sqinter},\\
\widetilde{\Phi}_i X^{\pm 1}_i&=X^{\pm 1}_{i+1}\widetilde{\Phi}_i, \widetilde{\Phi}_iX^{\pm 1}_{i+1}=X^{\pm 1}_i\tilde{\Phi}_i,
\widetilde{\Phi}_i X^{\pm 1}_l=X^{\pm 1}_l\widetilde{\Phi}_i \label{Xinter},\\
\widetilde{\Phi}_i C_i&=C_{i+1}\widetilde{\Phi}_i, \widetilde{\Phi}_i C_{i+1}=C_i \widetilde{\Phi}_i,
\widetilde{\Phi}_iC_l=C_l\widetilde{\Phi}_i \label{Cinter},\\
\widetilde{\Phi}_j \widetilde{\Phi}_i&=\widetilde{\Phi}_i \widetilde{\Phi}_j,
\widetilde{\Phi}_i\tilde{\Phi}_{i+1}\widetilde{\Phi}_i=\widetilde{\Phi}_{i+1}\widetilde{\Phi}_i , \widetilde{\Phi}_{i+1}\label{Braidinter}
\end{align}
for all admissible $i,j,l$ with $l\neq i, i+1$ and $|j-i|>1$. {\color{black} Observe that we can rewrite $\widetilde{\Phi}^2_i$ as 
$$
\widetilde{\Phi}^2_i=\mathsf{z}^4_i\epsilon^2\bigl(\frac{1}{\epsilon^2}-\frac{X_iX_{i+1}^{-1}}{(X_iX_{i+1}^{-1}-1)^2}-\frac{X_iX_{i+1}}{(X_iX_{i+1}-1)^2}\bigr). 
$$
Inspired by the above formula, }
{\color{black}for any pair of $(x,y)\in (\mathbb{K}^*)^2$, we consider the following condition 
\begin{align}\label{invertible}
	\frac{x^{-1}y}{(x^{-1}y-1)^2}+\frac{xy}{(xy-1)^2}=\frac{1}{\epsilon^2}.
\end{align}
According to \cite{JN}, via the substitution 
\begin{align}\label{substitute}
	x+x^{-1}=2\frac{qu+q^{-1}u}{q+q^{-1}}=\mathtt{q}(u),\qquad\qquad y+y^{-1}=2\frac{qv+q^{-1}v^{-1}}{q+q^{-1}}=\mathtt{q}(v)
\end{align} 
the condition \eqref{invertible}  is  equivalent to the condition which states that $u,v$ satisfy one of the following four equations 
\begin{align}\label{invertible2}
	v=q^2u,\quad v=q^{-2}u,\quad v=u^{-1},\quad v=q^{-4}u^{-1}, 
\end{align}
}
{\color{black} which will be useful later on in the construction of simple modules.}
\subsection{Cyclotomic Hecke-Clifford algebra $\mHfcn$}
To define the cycltomic Hecke-Clifford algebra $\mHfcn$, we need to take a $f=f(X_1)\in \mathbb{K}[X_1^\pm]$ satisfying \cite[(3.2)]{BK1}. Since we are working over algebraically closed field $\mathbb{K}$,  it is straightforward to check that $f(X_1)\in \mathbb{K}[X_1^\pm]$ satisfying \cite[(3.2)]{BK1} must be one of the following four forms:
$$\begin{aligned}
f^{(\mathsf{0})}_{\underline{Q}}&=\prod_{i=1}^m\biggl[\biggl(X_1+X^{-1}_1-\mathtt{q}(Q_i)\biggr)\biggr];\\
f^{\mathsf{(s)}}_{\underline{Q}}&=(X_1-1)\prod_{i=1}^m\biggl[\biggl(X_1+X^{-1}_1-\mathtt{q}(Q_i)\biggr)\biggr]=(X_1-1)f^{(\mathsf{0})}_{\underline{Q}};\\
f^{\mathsf{(ss})}_{\underline{Q}}&={\color{black} (X_1+1)(X_1-1)\prod_{i=1}^m\biggl[\bigl(X_1+X^{-1}_1-\mathtt{q}(Q_i)\biggr)\biggr]=(X_1+1)(X_1-1)f^{(\mathsf{0})}_{\underline{Q}};}\\
f^{\mathsf{(s')}}_{\underline{Q}}&=(X_1+1)\prod_{i=1}^m\biggl[\biggl(X_1+X^{-1}_1-\mathtt{q}(Q_i)\biggr)\biggr]=(X_1+1)f^{(\mathsf{0})}_{\underline{Q}},\\
\end{aligned}$$
for some $m\geq 0$ and $\underline{Q}=(Q_1,\cdots,Q_m)$ with $Q_1,\ldots,Q_m\in \mathbb{K}^*$. Here use the convention $f^{(\mathsf{0})}_{\underline{Q}}=1$ when $m=0$.

The non-degenerate cyclotomic Hecke-Clifford algebra $\mHfcn$ is defined as $$\mHfcn:=\mHcn/\mathcal{I}_f,
 $$ where $\mathcal{I}_f$ is the two sided ideal of $\mHcn$ generated by $f$. We shall denote the image of $X^{\alpha}, C^{\beta}, T_w$ in the cyclotomic quotient $\mHfcn$ still by the same symbol. Recall that degree of a Laurant polynomial $f=f(X_1)=\sum_{k=s}^ta_kX_1^k$ with $s\leq t$ and $a_s\neq 0, a_t\neq 0$ is $\deg(f)=t-s$.  Then we have the following due to \cite{BK1}. 
 \begin{lem}\cite[Theorem 3.6]{BK1}\label{lem:dim-Hf}
 	The set $\{X^{\alpha}C^{\beta}T_w~|~ \alpha\in\{0,1,\cdots,r-1\}^n,
 	\beta\in\mathbb{Z}_2^n, w\in {\mathfrak{S}_n}\}$ forms a basis of $\mHfcn$, where $r=\deg(f)$. 
 \end{lem}

Note that the map 
\begin{align*}
\tau: \mHcn\rightarrow \mHcn, \quad C_l\mapsto C_l, X_l\mapsto -X_l, T_j\mapsto T_j
\end{align*}
gives an algebra automorphism on $\mHcn$ and moreover $\tau(f^{\mathsf{(s')}}_{\underline{Q}})=f^{\mathsf{(s)}}_{\underline{-Q}}$, where $\underline{-Q}=(-Q_1,-Q_2,\ldots,-Q_m)$. This means the study of the representation theory of $\mHfcn$ for $f$ being of the form $f^{\mathsf{(s')}}_{\underline{Q}}$ is equivalent to that of $\mHfcn$ for $f$ being of the form $f^{\mathsf{(s)}}_{\underline{Q}}$. So it suffices to consider the first three situations. 


From now on, we fix $m\geq 0$ and $\underline{Q}=(Q_1,Q_2,\ldots,Q_m)\in(\mathbb{K}^*)^m$ and let $f=f^{\mathsf{(0)}}_{\underline{Q}}$ or $f=f^{\mathsf{(s)}}_{\underline{Q}}$ or $f=f^{\mathsf{(ss)}}_{\underline{Q}}$. Setting $r=\deg(f),$ then 
 $$\begin{aligned}
 f=\begin{cases}
     f^{\mathsf{(0)}}_{\underline{Q}}=\prod_{i=1}^m\biggl[\biggl(X_1+X^{-1}_1-\mathtt{q}(Q_i)\biggr)\biggr] , & \mbox{if $r=\deg(f)=2m$  };  \\
      f^{\mathsf{(s)}}_{\underline{Q}}=(X_1-1)\prod_{i=1}^m\biggl[\biggl(X_1+X^{-1}_1-\mathtt{q}(Q_i)\biggr)\biggr] , & \mbox{if $r=\deg(f)=2m+1$  };  \\
     f^{\mathsf{(ss)}}_{\underline{Q}} =(X_1+1)(X_1-1)\prod_{i=1}^m\biggl[\biggl(X_1+X^{-1}_1-\mathtt{q}(Q_i)\biggr)\biggr], & \mbox{if $r=\deg(f)=2m+2$ }.
    \end{cases} 
    \end{aligned}$$ 
We set $Q_0=Q_{0_+}=1, Q_{0_-}=-1$.
\begin{defn} Suppose  $\undla\in\mathscr{P}^{\bullet,m}_{n}$ with $\bullet\in\{\mathsf{0},\mathsf{s},\mathsf{ss}\}$ and $(i,j,l)\in \undla$, we define the residue of box $(i,j,l)$ with respect to the parameter $\undQ$ as follows:
\begin{equation}\label{eq:residue}
\res(i,j,l):=Q_lq^{2(j-i)}.
\end{equation}
If $\mathfrak{t}\in \Std(\undla)$ and $\mathfrak{t}(i,j,l)=a$, we set 
\begin{align}
        \res_\mathfrak{t}(a)&:=Q_lq^{2(j-i)};\label{resNon-dege-1}\\
\res(\mathfrak{t})&:=(\res_\mathfrak{t}(1),\cdots,\res_\mathfrak{t}(n)),\label{resNon-dege-2}\\
       \mathtt{q}(\res(\mathfrak{t}))&:=(\mathtt{q}(\res_{\mathfrak{t}}(1)), \mathtt{q}(\res_{\mathfrak{t}}(2)),\ldots, \mathtt{q}(\res_{\mathfrak{t}}(n))). \label{resNon-dege-3}
       \end{align}
\end{defn}
{\color{black} Recall the irreducible  $\mathcal{A}_n$-module $\mathbb{L}(\res(\mathfrak{t}))$ defined in \eqref{L-under-a} via $\underline{a}=\res(\mathfrak{t})$.} The following lemma follows directly from \eqref{resNon-dege-2} and Corollary \ref{lem:irrepAn}. 
\begin{lem}\label{lem:eigen-Xk}
Let  $\undla\in\mathscr{P}^{\bullet,m}_{n}$ with $\bullet\in\{\mathsf{0},\mathsf{s},\mathsf{ss}\}$. Suppose $\mathfrak{t}\in\Std(\undla)$. The eigenvalue of $X_k$ acting on the $\mathcal{A}_n$-module $\mathbb{L}(\res(\mathfrak{t}))$ is $\mathtt{b}_{\pm}(\res_{\mathfrak{t}}(k)$ for each $1\leq k\leq n$. 
Hence, the eigenvalue of $X_k+X_k^{-1}$ acting on the $\mathcal{A}_n$-module $\mathbb{L}(\res(\mathfrak{t}))$ is $\mathtt{q}(\res_{\mathfrak{t}}(k))$ for each $1\leq k\leq n$. 

\end{lem}

{\color{black}\subsection{Separate parameters.}

{\color{black}Recall the polynomial$P_{\mathscr{H}}(v,\underline{Q})$  introduced in \cite{Ar1}}. It's easy to check that $P_{\mathscr{H}}(v,\underline{Q})\neq 0$ if and only if {\color{black}the following holds for any $\undla\in\mathscr{P}^{m}_{n+1}$ and any $\mathfrak{t}\in\Std(\undla)$: 
\begin{equation}\label{Ar-separate}
\textbf{$\res_{\mathfrak{t}}(k)\neq\res_{\mathfrak{t}}(k+1)$} \begin{matrix}\textbf{ for any $k=1,\cdots,n.$ }
\end{matrix}
\end{equation}
In the rest of this section, analogous to \eqref{Ar-separate} we shall introduce a separate condition on the choice of the parameters $(q,\underline{Q})$} and $f=f^{(\bullet)}_{\underline{Q}}$ with $\bullet\in\{\mathtt{0},\mathtt{s},\mathtt{ss}\}$ and $r=\deg(f)$. Let $[1,n]:=\{1,2,\ldots,n-1\}$. 
\begin{defn}\label{defn:separate}
Let {\color{black}  $\bullet\in\{\mathsf{0},\mathsf{s},\mathsf{ss}\}$ and $\undQ=(Q_1,\ldots,Q_m)$.  Assume $\undla\in\mathscr{P}^{\bullet,m}_{n}$.} Then $(q,\undQ)$ is said to be {\em separate} with respect to $\undla$ if for any $\mathfrak{t}\in \undla$, the $\mathtt{q}$-sequence for $\mathfrak{t}$ defined via \eqref{resNon-dege-3} satisfy the following condition:
{\color{black}
$$
\mathtt{q}(\res_{\mathfrak{t}}(k))\neq\mathtt{q}(\res_{\mathfrak{t}}(k+1)) \text{ for any } k=1,\cdots,n-1. 
$$
}
\end{defn}

\begin{lem}\label{separate in residue}
Let {\color{black}  $\bullet\in\{\mathsf{0},\mathsf{s},\mathsf{ss}\}$ and $\undQ=(Q_1,\ldots,Q_m)$.  Assume $\undla\in\mathscr{P}^{\bullet,m}_{n}$.} Then $(q,\undQ)$ is separate with respect to $\undla$ if and only if for any $\mathfrak{t}\in \undla$ and $k=1,\cdots,n-1,$ 
{\color{black}
 $$
 \res_{\mathfrak{t}}(k)\neq \res_{\mathfrak{t}}(k+1)\text{ and }\res_{\mathfrak{t}}(k)\res_{\mathfrak{t}}(k+1)q^2\neq 1.
 $$}
\end{lem}

\begin{proof}
By \eqref{substitution0}, we have $\mathtt{q}(x)=\mathtt{q}(y)$ if and only if $x=y$ or $xyq^2=1$. This proves the Lemma.
	\end{proof}

Recall that $\undQ=(Q_1,\ldots,Q_m)\in (\mathbb{K}^*)^n$ and  $\pm 1\neq q \in \mathbb{K}^*$. Then for any $n\in \N$, we define $P^{\bullet}_{n}(q^2,\undQ)$ as follows:
    
         $$\begin{aligned}
 P_n^{(\bullet)}(q^2,\undQ):=\begin{cases}
    \prod\limits_{t=1}^{n}\bigl(q^{2t}-1\bigr)\prod\limits_{i=1}^m\biggl(\prod\limits_{t=3-n}^{n-1}\bigl(Q^2_i-q^{-2t}\bigr)\prod\limits_{t=1-n}^{n}\bigl(Q^2_i-q^{-4t}\bigr)\biggr)&\\
     \cdot \prod\limits_{1\leq i<i'\leq m}\biggl(\prod\limits_{t=1-n}^{n-1}\bigl({Q_i}-Q_{i'}q^{-2t}\bigr)
   \bigl(Q_iQ_{i'}-q^{-2(t+1)}\bigr)\biggr), \quad \mbox{if $\bullet=\mathsf{0}$ };  \\
    &\\
        \prod\limits_{t=1}^{n}\biggl(\bigl(q^{2t}-1\bigr)\bigl(q^{2t}+1\bigr)\biggr)\prod\limits_{i=1}^m\biggl(\prod\limits_{t=3-n}^{n-1}\bigl(Q^2_i-q^{-2t}\bigr)\prod\limits_{t=1-n}^{n}\bigl(Q^2_i-q^{-4t}\bigr)\biggr)&\\
      \cdot \prod\limits_{1\leq i<i'\leq m}\biggl(\prod\limits_{t=1-n}^{n-1}\bigl({Q_i}-Q_{i'}q^{-2t}\bigr)
      \bigl(Q_iQ_{i'}-q^{-2(t+1)}\bigr)\biggr),  \quad \mbox{if $\bullet=\mathsf{s}$ or  $\mathsf{ss},$}\\
    \end{cases} 
    \end{aligned}$$ where for  $n=1,$  the product $\prod\limits_{t=3-n}^{n-1}\bigl(Q^2_i-q^{-2t}\bigr)$ is understood to be $1$.

\begin{prop}\label{separate formula} 
Let $n\geq 1,\,m\geq 0$,  $\undQ=(Q_1,\ldots,Q_m)$ and  $\bullet\in\{\mathsf{0},\mathsf{s},\mathsf{ss}\}$. Then $(q,\undQ)$ is separate with respect to $\underline{\mu}$ for any  $\underline{\mu}\in\mathscr{P}^{\bullet,m}_{n+1}$ if and only if $P_n^{(\bullet)}(q^2,\undQ)\neq 0$.
\end{prop}
\begin{proof}
We assume $n>1$. {\color{black}  In the case $\bullet=\mathsf{0}$, by Lemma \ref{separate in residue} it is straightforward to compute that $(q,\undQ)$ is separate with respect to $\underline{\mu}$  for any $\underline{\mu}\in\mathscr{P}^{\mathsf{0},m}_{n+1}$  if and only if}
 \begin{align*}
           &\bigl((q^2)^{t}\bigr)\neq 1,\quad \forall 1\leq t\leq n;\\
           &\bigl(Q^2_i(q^2)^{t}\bigr)\neq 1,\quad \forall 3-n\leq t\leq n-1,\,1\leq i\leq m,; \\  
           &\bigl(Q^2_i(q^2)^{2t}\bigr)\neq 1,\quad \forall 1-n\leq t\leq n,\,1\leq i\leq m,; \\            
                            &\bigl(\frac{Q_i}{Q_{i'}}(q^2)^{t}\bigr) \neq 1  ,\quad \forall 1-n\leq t\leq n-1,\,1\leq i\neq i'\leq m;   \\
                             &\bigl(Q_iQ_{i'}(q^2)^{t}\bigr)\neq 1 ,\quad \forall 2-n\leq t\leq n,\,1\leq i\neq i'\leq m. 
\end{align*} 
Meanwhile, {\color{black}  in the case $\bullet=\mathsf{s}$, by Lemma \ref{separate in residue} it is straightforward to compute that $(q,\undQ)$ is separate with respect  to $\underline{\mu}$ for  any $\underline{\mu}\in\mathscr{P}^{\mathsf{s},m}_{n+1}$ if and only if}
\begin{align*}
           &\bigl((q^2)^{t}\bigr)\neq 1,\quad \forall 1\leq t\leq n;\\
 &\bigl((q^2)^{2t}\bigr)\neq 1,\quad \forall 1\leq t\leq n;\\
            &\bigl(Q_i(q^2)^{t}\bigr)\neq 1,\quad \forall 1-n\leq t\leq n,\,1\leq i\leq m; \\
              &\bigl(Q^2_i(q^2)^{t}\bigr)\neq 1,\quad \forall 3-n\leq t\leq n-1,\,1\leq i\leq m,; \\  
           &\bigl(Q^2_i(q^2)^{2t}\bigr)\neq 1,\quad \forall 1-n\leq t\leq n,\,1\leq i\leq m,; \\       
                            &\bigl(\frac{Q_i}{Q_{i'}}(q^2)^{t}\bigr) \neq 1  ,\quad \forall 1-n\leq t\leq n-1,\,1\leq i\neq i'\leq m;   \\
                             &\bigl(Q_iQ_{i'}(q^2)^{t}\bigr)\neq 1 ,\quad \forall 2-n\leq t\leq n,\,1\leq i\neq i'\leq m
                                              \end{align*}
In addition,   {\color{black}  in the case $\bullet=\mathsf{ss}$, by Lemma \ref{separate in residue} it is straightforward to compute that $(q,\undQ)$ is separate with respect  to $\underline{\mu}$ for  any $\underline{\mu}\in\mathscr{P}^{\mathsf{ss},m}_{n+1}$ if and only if}
    \begin{align*}
    	&\bigl(\pm (q^2)^{t}\bigr)\neq 1,\quad \forall 1\leq t\leq n;\\
    	 &\bigl((q^2)^{2t}\bigr)\neq 1,\quad \forall 1\leq t\leq n;\\
    	 &\bigl(\pm Q_i(q^2)^{t}\bigr)\neq 1,\quad \forall 1-n\leq t\leq n,\,1\leq i\leq m; \\
    	  &\bigl(Q^2_i(q^2)^{t}\bigr)\neq 1,\quad \forall 3-n\leq t\leq n-1,\,1\leq i\leq m,; \\  
    	&\bigl(Q^2_i(q^2)^{2t}\bigr)\neq 1,\quad \forall 1-n\leq t\leq n,\,1\leq i\leq m,; \\       
    	&\bigl(\frac{Q_i}{Q_{i'}}(q^2)^{t}\bigr) \neq 1  ,\quad \forall 1-n\leq t\leq n-1,\,1\leq i\neq i'\leq m;   \\
    	&\bigl(Q_iQ_{i'}(q^2)^{t}\bigr)\neq 1 ,\quad \forall 2-n\leq t\leq n,\,1\leq i\neq i'\leq m
    \end{align*}

                                            The case $n=1$ can be checked similarly by observing that  the range sets for $t$ in some of the inequalities are slightly different.  Then the Proposition follows from a direct computation. 
\end{proof}



We shall use the following observation repatedly.

\begin{lem}\label {important condition}
	Let $\undQ=(Q_1,\ldots,Q_m)$ and  $\bullet\in\{\mathsf{0},\mathsf{s},\mathsf{ss}\}$. Suppose $P_n^{(\bullet)}(q^2,\undQ)\neq 0$. Then for any $n'\leq n$ and any  $\undla'\in\mathscr{P}^{\bullet,m}_{n'}$, $(q,\undQ)$ is separate with respect to $\undla'$. 
	\end{lem}

\begin{proof}
	Note that any $\mathfrak{t'}\in\Std(\undla)$ can be embedded into some  $\mathfrak{t}\in\Std(\underline{\mu})$, where $\underline{\mu}\in\mathscr{P}^{\bullet,m}_{n+1}$. Since $P_n^{(\bullet)}(q^2,\undQ)\neq 0$, by Proposition \ref{separate formula}, $(q,\undQ)$ is separate with respect to $\underline{\mu}$  and hence  the Lemma follows as $\mathfrak{t'}\subset\mathfrak{t}$.
	\end{proof}
The following is key for our construction in the next section. 

\begin{lem}\label {important condition1}
	Let $\undQ=(Q_1,\ldots,Q_m)$ and  $\bullet\in\{\mathsf{0},\mathsf{s},\mathsf{ss}\}$. Suppose $P_n^{(\bullet)}(q^2,\undQ)\neq 0$. Then for any $\undla\in\mathscr{P}^{\bullet,m}_{n}$ and any $\mathfrak{t}\in\Std(\undla)$, we have the following
	\begin{enumerate}
		\item  $\mathtt{b}_{\pm}(\res_{\mathfrak{t}}(k))\neq \pm 1$ for $k\notin \mathcal{D}_{\mathfrak{t}}$;
		\item   $\mathtt{q}(\res_{\mathfrak{t}}(k))\neq \mathtt{q}(\res_{\mathfrak{t}}(k+1))$ for $k=1,\cdots,n-1$;
		\item  $\res_{\mathfrak{t}}(k)$ and $\res_{\mathfrak{t}}(k+1)$ does not satisfy any one of the four equations in \eqref{invertible2} {\color{black}if $k,k+1$ are not in the adjacent diagonals of $\mathfrak{t}.$}
	\end{enumerate}
\end{lem}

\begin{proof}
	
	(1)  Suppose $\alpha=\mathfrak{t}^{-1}(k)\in \undla$ satisfying $\alpha\notin \mathcal{D}_{\undla}$ and $\mathtt{b}_{\pm}(\res(\alpha))=\pm1$. By \eqref{substitution0}, we have $\mathtt{q}(\res(\alpha))=\pm2$. Hence, $\res(\alpha)^2=1$ or $\res(\alpha)^2=q^{-4}$. That is, $\res(\alpha)(\res(\alpha)q^{-2})q^2=1$ or $\res(\alpha)(\res(\alpha)q^2)q^2=1$. If $\res(\alpha)(\res(\alpha)q^{-2})q^2=1$,  then we claim that there is no additive node in $\mathfrak{t}\downarrow_k$ below  $\alpha$. Otherwise, we can add this node to $\mathfrak{t}\downarrow_k$  and label it by $k+1$. We denote this new tableau by $\mathfrak{t'}$. Now we have $\res_k(\mathfrak{t'})\res_{k+1}(\mathfrak{t'})q^{2}=\res(\alpha)(\res(\alpha)q^{-2})q^2=1$, which contradicts to  Lemma \ref{separate in residue} and Lemma \ref{important condition}. Combing together with $\alpha\notin \mathcal{D}_{\undla}$, we deduce that there is no node below the node $\alpha'$ which is exactly on the left of $\alpha$. Hence we can reconstruct a new tableau $\mathfrak{t''}$ such that $\alpha=\mathfrak{t''}^{-1}(k)$ and $\alpha'=\mathfrak{t''}^{-1}(k-1)$. Now we have $\res_k(\mathfrak{t''})\res_{k-1}(\mathfrak{t''})q^{2}=\res(\alpha)(\res(\alpha)q^{-2})q^2=1$, which again contradicts to Lemma \ref{separate in residue} and Lemma \ref{important condition}. 
%
	  If $\res(\alpha)(\res(\alpha)q^2)q^2=1$, then we can also derive contradiction in a similar way as one can show that in this case there is no addative node in $\mathfrak{t}\downarrow_k$ on the right of $\alpha$. Then one can reconstruct a new tableau and eventually this results in a contradiction  to the Lemma \ref{separate in residue} and Proposition \ref{separate formula}.
	
	
	(2) This follows from Lemma \ref{separate in residue} and Lemma \ref{important condition}.
	
	(3) Suppose $\alpha_1=\mathfrak{t}^{-1}(k),\,\alpha_2=\mathfrak{t}^{-1}(k+1)\in\undla$ are not in the adjacent diagonals of $\undla$. {\color{black} If $\res(\alpha_1)=q^2\res(\alpha_2)$, we claim that in $\mathfrak{t}\downarrow_{k+1}$, there {\color{black}exists} no additive node on the right-hand-side of $\alpha_2$. Otherwise, we can add this node to $s_k(\mathfrak{t}\downarrow_{k+1})$ and label it by $k+2$. We denote this new tableau by $\mathfrak{t'}$. Then in $\mathfrak{t'}$, $\res_{k+1}(\mathfrak{t'})=\res(\alpha_1)=q^2\res(\alpha_2)=\res_{k+2}(\mathfrak{t'})$. This contradicts to Lemma \ref{separate in residue} and Lemma \ref{important condition}. Hence, in $\mathfrak{t}\downarrow_{k}$, there is a removable node $\alpha_3$ above $\alpha_2$. Note that $\alpha_1$ is also a removable node in $\mathfrak{t}\downarrow_{k}$. This means we can reconstruct a new tablaeu $\mathfrak{t''}$ such that $\alpha_2=\mathfrak{t''}^{-1}(k)$ and $\alpha_3=\mathfrak{t''}^{-1}(k-1)$. Now in $\mathfrak{t''}$, we have $\res_{k}(\mathfrak{t'})=\res(\alpha_1)=q^2\res(\alpha_2)=\res_{k-1}(\mathfrak{t''})$, which again contradicts to Lemma \ref{separate in residue} and Lemma \ref{important condition}. 
		
	The same arguement applies to the case $\res(\alpha_1)=q^{-2}\res(\alpha_2)$ as well as $\res(\alpha_1)=q^{-4}\res(\alpha_2)^{-1}$. For the case $\res(\alpha_1)=\res(\alpha_2)^{-1}$, we rewrite it as $\res(\alpha_1)(\res(\alpha_2)q^{-2})q^2=1$. 
In the case $\alpha_1\notin\mathcal{D}_{\undla}$ or $\alpha_2\notin\mathcal{D}_{\undla}$, we can apply the argument  similar to the proof of (1) to deduce a contradiction to Lemma \ref{separate in residue} and Proposition \ref{separate formula}.
 Otherwise, $\alpha_1\in\mathcal{D}_{\undla}$ and $\alpha_2\in\mathcal{D}_{\undla}$. In this case, we have $\bullet=\mathsf{ss}$ and $\res(\alpha_1)\res(\alpha_2)=-1$ which is contradicts to $\res(\alpha_1)=\res(\alpha_2)^{-1}$. 
 Putting together, we obtain that $(\res(\alpha_1),\,\res(\alpha_2))$ does not satisfy any one of the four equations in \eqref{invertible2}.	
	}
	
\end{proof}

\begin{lem}\label{lem:action property-1}
	Let $\undQ=(Q_1,\ldots,Q_m)$ and  $\bullet\in\{\mathsf{0},\mathsf{s},\mathsf{ss}\}$. Suppose $P_n^{(\bullet)}(q^2,\undQ)\neq 0$.  Let $\undla\in\mathscr{P}^{\bullet,m}_{n}$, then 
	any pair of  $(a_1,a_2)$ with $a_1,a_2$ being the eigenvalues of $X_{k}$ and $X_{k+1}$ on $\mathbb{L}(\res(\mathfrak{t}))$, respectively, does not satisfy \eqref{invertible}, for any $\mathfrak{t}\in\Std(\undla)$ and $k,k+1$ {\color{black} being not in the adjacent diagonals  }of $\mathfrak{t}$. 
\end{lem}

\begin{proof}
	Suppose $\mathfrak{t}\in\Std(\undla)$.  Fix any $k_1,k_2\in[1,n]$ such that $\alpha_1=(\mathfrak{t}^{\undla})^{-1}(k_1),\alpha_2=(\mathfrak{t}^{\undla})^{-1}(k_2)$ are not in the adjacent diagonals.  
	Let $a_1,a_2$ be eigenvalues of $X_{k_1}$ and $X_{k_2}$ acting on $\mathbb{L}(\res(\mathfrak{t}))$. 
		By Lemma \ref{lem:eigen-Xk} we have $a_1=\mathtt{b}_\pm(\res(\alpha_1)) $ and $a_2=\mathtt{b}_\pm(\res(\alpha_2))$. That is, $a_1+a_1^{-1}=\mathtt{q}(\res(\alpha_1))$ and $a_2+a_2^{-1}=\mathtt{q}(\res(\alpha_2))$. 
		Then by the fact that  \eqref{invertible} is equivalent to \eqref{invertible2} via the substitution \eqref{substitute}, we obtain that  the pair $(a_1,a_2)$ does not satisfy \eqref{invertible} by Lemma \ref{important condition1}(3). 
\end{proof}

The following lemma will be useful in the subsequent section. 

\begin{lem}\label{lem:different residues}
Let $m\geq 0,\,n\geq 1$, $\undQ=(Q_1,\ldots,Q_m)\in(\mathbb{K}^*)^m$ and $\bullet\in\{\mathsf{0},\mathsf{s},\mathsf{ss}\}$. Suppose $P_n^{(\bullet)}(q^2,\undQ)\neq 0$.  
Then for any $\undla,\,\underline{\mu}\in\mathscr{P}^{\bullet,m}_{n},\,\mathfrak{t}\in\Std(\undla),\,\mathfrak{t'}\in\Std(\underline{\mu})$, we have  $\mathtt{q}(\res(\mathfrak{t}))\neq \mathtt{q}(\res(\mathfrak{t}'))$ if $\mathfrak{t}\neq \mathfrak{t'}$.
\end{lem}

\begin{proof}
Let $k<n$ be the maximal integer such that $\mathfrak{t}\downarrow_k=\mathfrak{t'}\downarrow_k$ but $ \mathfrak{t}\downarrow_{k+1}\neq\mathfrak{t'}\downarrow_{k+1}$.   Observe that $\mathfrak{t}^{-1}(k+1)$ and $\mathfrak{t'}^{-1}(k+1)$ are two different additive node for the shape of $\mathfrak{t}\downarrow_k=\mathfrak{t'}\downarrow_k$. Adding these two nodes to $\mathfrak{t}\downarrow_k$ and labelling them by $k+1,k+2$, respectively, one can obtain a standard tableau $\mathfrak{s}$ of some shape $\underline{\gamma}\in \mathscr{P}^{\bullet,m}_{k+2}$ with $k+2\leq n+1$. Then apply Lemma \ref{separate in residue} and Lemma \ref{important condition}, we deduce that $\mathtt{q}(\res_{k+1}(\mathfrak{t}))\neq \mathtt{q}(\res_{k+1}(\mathfrak{t'}))$.   This proves the Lemma.
\end{proof}

\begin{example}\label{deform} 
	
 When $q,Q_1,\ldots,Q_m$ are algebraically independent over $\Z$, and $\mathbb{F}$ is the algebraic closure of $\mathbb{Q}(q,Q_1,\ldots,Q_m)$, i.e., for generic non-degenerate cyclotomic Hecke-Clifford algebra, the separate condition clearly holds by Proposition \ref{separate formula}.
\end{example}
}
\section{Semi-simplicity on  non-degenerate cyclotomic Hecke-Clifford superalgebras}

\subsection{Construction of Simple modules}

{\bf For this subsection, we shall fix the parameter $\undQ=(Q_1,Q_2,\ldots,Q_m)\in(\mathbb{K}^*)^m$ and  $f=f^{(\bullet)}_{\undQ}$ with $\bullet\in\{\mathsf{0},\mathsf{s},\mathsf{ss}\}$.  Accordingly, we define the residue of boxes in the young diagram $\undla$ via \eqref{eq:residue} as well as $\res(\mathfrak{t})$ for each $\mathfrak{t}\in\Std(\undla)$ with $\undla\in\mathscr{P}^{\bullet,m}_{n}$ with $m\geq 0$.}

\begin{defn}\label{defn:admissible}
Let {\color{black}$\bullet\in\{\mathsf{0},\mathsf{s},\mathsf{ss}\}$ and $\undla\in\mathscr{P}^{\bullet,m}_{n}$.} Suppose $\mathfrak{t}\in\Std(\undla)$ and $1\leq l\leq n$. 
If  $s_l\cdot\mathtt{q}(\res(\mathfrak{t}))=\mathtt{q}(\res(\mathfrak{u}))$ for some $\mathfrak{u}\in \Std(\undla)$, then the simple transposition $s_l$ is said to be admissible with respect to the sequence $\mathtt{q}(\res(\mathfrak{t}))$. 
\end{defn}
\begin{lem}\label{lem:admissible-residue}
{\color{black}Let $\bullet\in\{\mathsf{0},\mathsf{s},\mathsf{ss}\}$, $\undla\in\mathscr{P}^{\bullet,m}_{n}$ and $\mathfrak{t}\in\Std(\undla)$. Suppose $(q,\undQ)$ is separate with respect to $\undla$. }Then $s_l$ is admissible with respect to $\mathfrak{t}$ if and only if $s_l$ is admissible with respect to $\mathtt{q}(\res(\mathfrak{t}))$ for $1\leq l\leq n-1$.  
\end{lem}
\begin{proof}
If $s_l$ is admissible with respect to $\mathfrak{t}$, then $s_l\cdot\mathfrak{t}\in\Std(\undla)$ and moreover $\mathtt{q}(\res(s_l\cdot\mathfrak{t}))=s_l\cdot\mathtt{q}(\res(\mathfrak{t}))$ which means $s_l$ is admissible with respect to $\mathtt{q}(\res(\mathfrak{t}))$. Conversely, if $s_l$ is admissible with respect to $\res(\mathfrak{t})$, then we have 
\begin{equation}\label{eq:standard-u}
\mathtt{q}(\res(s_l\cdot\mathfrak{t}))=s_l\cdot\mathtt{q}(\res(\mathfrak{t}))=\mathtt{q}(\res(\mathfrak{u}))
\end{equation} 
for some $\mathfrak{u}\in \Std(\undla)$. {\color{black}We claim that $s_l\cdot\mathfrak{t}$ is standard. Otherwise, $l,\,l+1$ are in the same row or in the same column of $\mathfrak{t}$. Suppose $l,\,l+1$ are in the same row.  Firstly, we have $\mathtt{q}(\res(\mathfrak{t}\downarrow_{l-1}))=\mathtt{q}(\res(s_l\cdot\mathfrak{t}\downarrow_{l-1}))=\mathtt{q}(\res(\mathfrak{u}\downarrow_{l-1}))$ and both of $\mathfrak{t}\downarrow_{l-1}$ and $\mathfrak{u}\downarrow_{l-1}$ are standard. Thus by Lemma  \ref{lem:different residues} we have $\mathfrak{t}\downarrow_{l-1}=\mathfrak{u}\downarrow_{l-1}$. Moreover by Lemma \ref{separate in residue}, we obtain $\mathtt{q}(\res_l(\mathfrak{t}))\neq \mathtt{q}(\res_{l+1}(\mathfrak{t}))=\mathtt{q}(\res_l(\mathfrak{u}))$. This implies that $\alpha=\mathfrak{u}^{-1}(l)$ and $\alpha'=\mathfrak{t}^{-1}(l)$ are two different additive nodes of $\mathfrak{t}\downarrow_{l-1}=\mathfrak{u}\downarrow_{l-1}$. Now we can reconstruct a new tableau $\mathfrak{t'}$ such that $\mathfrak{t'}^{-1}(l)=\alpha',\mathfrak{t'}^{-1}(l+1)=\alpha,\,\mathfrak{t'}^{-1}(l+2)=\mathfrak{t}^{-1}(l+1)$. Then in $\mathfrak{t'}$, we have  $\mathtt{q}(\res_{l+1}(\mathfrak{t'}))= \mathtt{q}(\res_{l}(\mathfrak{u}))=\mathtt{q}(\res_{l+1}(\mathfrak{t}))=\mathtt{q}(\res_{l+2}(\mathfrak{t'}))$. This contradicts to Lemma \ref{separate in residue} and Lemma \ref{important condition} since $l+2\leq n+1$. Similarly argument applies if $l,\,l+1$ are in the same row. Hence $s_l\cdot\mathfrak{t}$ is standard.  Then $s_l\cdot\mathfrak{t}=\mathfrak{u}$ by Lemma \ref{lem:different residues}.}
\end{proof}

\begin{defn}{\color{black}For $\bullet\in\{\mathsf{0},\mathsf{s},\mathsf{ss}\}$ and $\undla\in\mathscr{P}^{\bullet,m}_{n}$}, we define the $\mathcal{A}_n$-module
$$
\mathbb{D}(\undla):=\oplus_{\tau\in P(\undla)}\mathbb{L}(\res(\mathfrak{t}^{\undla}))^{\tau}. 
$$ 
\end{defn}
{\bf In the remaining part of this section, we shall fix  $\bullet\in\{\mathtt{0},\mathtt{s},\mathtt{ss}\}$ and assume that  the parameters $q$ and $\undQ=(Q_1,Q_2,\ldots,Q_m)\in(\mathbb{K}^*)^m$ satisfy $P^{(\bullet)}_{n}(q^2,\undQ)\neq 0$.} {\color{black}By
 Lemma \ref{important condition1} (1) and \eqref{eq:residue}, we deduce $\{k|1\leq k\leq n, (\res_\mathfrak{t^{\undla}}(k))\sim \pm 1\}=\mathcal{D}_{\mathfrak{t}^{\undla}}$ and }
\begin{equation}\label{diagonal-D}
\sharp\mathcal{D}_{\mathfrak{t}^{\undla}}=\sharp\mathcal{D}_{\undla}
=\left\{
\begin{array}{ll}
0,&\text{if }\undla=(\lambda^{(1)},\ldots,\lambda^{(m)})\in\mathscr{P}^{\mathsf{0},m}_{n},\\
\ell(\lambda^{(0)}),&\text{if }\undla=(\lambda^{(0)},\lambda^{(1)},\ldots,\lambda^{(m)})\in\mathscr{P}^{\mathsf{s},m}_{n},\\
\ell(\lambda^{(0_-)})+\ell(\lambda^{(0_+)}),&\text{if }\undla=(\lambda^{(0_-)},\lambda^{(0_+)}, \lambda^{(1)},\ldots,\lambda^{(m)})\in\mathscr{P}^{\mathsf{ss},m}_{n}.
\end{array}
\right. 
\end{equation}
 Hence, by Corollary \ref{lem:irrepAn} we have 
 \begin{equation}\label{eq:dimDla}
 \text{dim}~\mathbb{D}(\undla)=
  2^{n-\lfloor\frac{\sharp \mathcal{D}_{\undla}}{2}\rfloor}\cdot |\Std(\undla)|.
 \end{equation}

The following is due to  Remark \ref{rem:Ltau1} and Lemma \ref{lem:eigen-Xk}.
\begin{lem}\label{lem:action property-2}
		{\color{black}Let $\bullet\in\{\mathtt{0},\mathtt{s},\mathtt{ss}\}$ and $\undla\in\mathscr{P}^{\bullet,m}_{n}$.} The eigenvalue of $X_k$ acting on the $\mathcal{A}_n$-module $\mathbb{L}(\res(\mathfrak{t}^{\undla}))^{\tau}$  is $\mathtt{b}_{\pm}(\res_{\tau\cdot\mathfrak{t}^{\underline{\lambda}}}(k))$ for each $1\leq k\leq n$. 
		Hence, the eigenvalue of $X_k+X^{-1}_k$ acting on the $\mathcal{A}_n$-module $\mathbb{L}(\res(\mathfrak{t}^{\undla}))^{\tau}$  is $\mathtt{q}(\res_{\tau\cdot\mathfrak{t}^{\underline{\lambda}}}(k))$ for each $1\leq k\leq n$. 
		\end{lem}

\begin{proof}
By \eqref{resNon-dege-2}  and Remark \ref{rem:Ltau1} we have 
\begin{align*}
\mathbb{L}(\res_{\mathfrak{t}^{\undla}})^\tau &\cong\mathbb{L}(\res_{\mathfrak{t}^{\undla}}(\tau^{-1}(1)), \cdots, \res_{\mathfrak{t}^{\undla}}(\tau^{-1}(n))) \\
&=\mathbb{L}(\res_{\tau\cdot\mathfrak{t}^{\undla}}(1),\ldots, \res_{\tau\cdot\mathfrak{t}^{\undla}}(n))
\end{align*}

This is due to the fact that the node occupied by $k$ in $\tau\cdot\mathfrak{t}^{\undla}$ coincides with the node occupied by $\tau^{-1}(k)$ in $\mathfrak{t}^{\undla}$ for each $1\leq k\leq n$.  In other words, we have 
\begin{equation}
\tau\cdot\res(\mathfrak{t}^{\undla})=\res(\tau\cdot \mathfrak{t}^{\undla}). 
\end{equation}
Then the lemma follows by Lemma \ref{lem:eigen-Xk}. 
\end{proof}

To define a $\mHfcn$-module structure  on $\mathbb{D}(\undla)$, we introduce two operators on $\mathbb{L}(\res(\mathfrak{t}^{\undla}))^{\tau}$ for each $\tau\in P(\undla)$ in the following as a generalization of the operators in \cite{Wa}: 
\begin{align}
\widetilde{\Xi}_i u&:=\bigl(-\epsilon\frac{1}{X_i X^{-1}_{i+1}-1}+\epsilon\frac{1}{X_i X_{i+1}-1}C_i C_{i+1}\bigr)\, u, \label{Operater1Non-dege}\\
\widetilde{\Omega}_i u&:=\sqrt{1-\epsilon^2 \biggl(\frac{X_iX^{-1}_{i+1}}{(X_iX^{-1}_{i+1}-1)^2}
+\frac{X^{-1}_iX^{-1}_{i+1}}{(X^{-1}_iX^{-1}_{i+1}-1)^2}\biggr)}\, u\label{Operater2Non-dege},
\end{align} where $u\in \mathbb{L}(\res(\mathfrak{t}^{\undla}))^{\tau}$.
By the second part of Lemma \ref{important condition1} and Lemma \ref{lem:action property-2}, the eigenvalues of $X_i+X^{-1}_i$ and $X_{i+1}+X^{-1}_{i+1}$ on $\mathbb{L}(\res(\mathfrak{t}^{\undla}))^{\tau}$ are different, hence the operators $\widetilde{\Xi}_i$ and $\widetilde{\Omega}_i$  are well-defined on $\mathbb{L}(\res(\mathfrak{t}^{\undla}))^{\tau}$ for each $\tau\in P(\undla)$. 

\begin{thm}\label{Construction}
{\color{black}Let $\bullet\in\{\mathsf{0},\mathsf{s},\mathsf{ss}\}$ and $\undQ=(Q_1,\ldots,Q_m)$. Suppose $f=f^{(\bullet)}_{\undQ}(X_1)$ and  $P^{\bullet}_{n}(q^2,\undQ)\neq 0$. }
Then  $\mathbb{D}(\undla)$ affords a $\mHfcn$-module via
\begin{align}
T_iz^{\tau}= \left \{
 \begin{array}{ll}
 \widetilde{\Xi}_i z^{\tau}
 +\widetilde{\Omega}_i z^{s_i\tau},
 & \text{ if } s_i \text{ is admissible with respect to } \tau\cdot \res(\mathfrak{t}^{\undla}), \\
 \widetilde{\Xi}_i z^{\tau}
 , & \text{ otherwise},
 \end{array}
 \right.\label{actionformulaNon-dege}
\end{align}
 for any $1\leq i\leq n-1,\,z\in \mathbb{L}(\res(\mathfrak{t}^{\undla}))$ and $\tau\in P(\undla)$.
\end{thm}

\begin{proof}

 Fix $\tau\in P(\undla)$ and any $z^\tau\in \mathbb{L}(\res(\mathfrak{t}^{\undla}))^{\tau}$ with $z\in \mathbb{L}(\res(\mathfrak{t}^{\undla}))$. Since the action of $X^\pm_1,\cdots,X^\pm_n$ are semi-simple on $\mathbb{L}(\res(\mathfrak{t}^{\undla}))^{\tau}$, to show that $\mathbb{D}(\undla)$ affords a $\mHfcn$-module via \eqref{actionformulaNon-dege},  it suffices to show that the actions of $T_i, C_j, X_j$ on $z^\tau$ satisfy the relations \eqref{Braid}-\eqref{PC} and moreover the polynomial $f(X_1)$ satisfies $f(X_1)z^\tau=0$ in the case $z^\tau$ is a simultaneous eigenvector of $X^\pm_1,\cdots,X^\pm_n$ {\color{black} for $f(X_1)=f^{(\mathtt{\bullet})}_{\undQ}(X_1)$}. From now on, we assume $z^\tau$ is a simultaneous eigenvector of $X^\pm_1,\cdots,X^\pm_n$. 
  
Firstly,  by Lemma \ref{lem:action property-2}, 
\begin{equation}\label{f-action}
(X_1+X_1^{-1})z^\tau=\mathtt{q}(\res_{\tau\cdot\mathfrak{t}^{\underline{\lambda}}}(1))z^\tau
\end{equation}
Since $\tau\in P(\undla)$, we have that $\tau\cdot\mathfrak{t}^{\underline{\lambda}}$ is standard and hence the box occupied by the number $1$ must be at the position $(1,1)$ in one component of Young diagrams $\undla$. This means $\res_{\tau\cdot\mathfrak{t}^{\underline{\lambda}}}(1)=\pm 1$ or $\res_{\tau\cdot\mathfrak{t}^{\underline{\lambda}}}(1)=Q_t$ for some $1\leq t\leq m$. In the case $\res_{\tau\cdot\mathfrak{t}^{\underline{\lambda}}}(1)=\pm 1$, by \eqref{f-action} we have $X_1z^\tau=\pm z^\tau$. This together with \eqref{f-action} leads to 
$$
f(X_1)z^\tau=0
$$
for $f(X_1)=f^{(\mathtt{\bullet})}_{\undQ}(X_1)$. 

 As each  $\mathbb{L}(\res(\mathfrak{t}^{\undla}))^{\tau}$ is a $\mathcal{A}_n$-module, it remains to check the actions of  $T_i, C_j, X_j$ on $z^\tau$ satisify relations \eqref{Braid}, \eqref{PX1}, \eqref{PX2}, \eqref{PX3} and \eqref{PC}. Write $\tau\cdot \res(\mathfrak{t}^{\undla})=\res(\tau\cdot \mathfrak{t}^{\undla})=(\iota_1,\ldots,\iota_n)$. Then 
 \begin{equation}\label{Xi-ztau}
 X_i z^\tau=a_iz^\tau,\qquad a_i=\mathtt{b}_{\pm}(\mathtt{q}(\iota_i)) \text{ for each } 1\leq i\leq n. 
 \end{equation}
  
 {\bf Relations \eqref{Braid}}.  It is straightforward to check that $ T_iT_j z^\tau=T_jT_i z^\tau$ holds in the case $|i-j|>1$ by \eqref{actionformulaNon-dege}. Let $1\leq i\leq n-1$. Then by \eqref{actionformulaNon-dege} if  $s_i$ is admissible with respect to $\tau\cdot \res(\mathfrak{t}^{\undla})$, we can compute 
 \begin{align*}
       T^2_iz^{\tau}&=\widetilde{\Xi}_i^2 z^\tau+\widetilde{\Xi}_i\widetilde{\Omega}_iz^{s_i\tau}+\widetilde{\Omega}_i(\widetilde{\Xi}_iz^\tau)^{s_i}
     +\widetilde{\Omega}_i(\widetilde{\Omega}_i z^{s_i\tau})^{s_i}\\
       &=\widetilde{\Xi}_i^2 z^\tau+\widetilde{\Omega}_i\widetilde{\Xi}_iz^{s_i\tau}+\widetilde{\Omega}_i(\widetilde{\Xi}_iz^\tau)^{s_i}
     +\widetilde{\Omega}^2_i z^\tau\\
    & =(\widetilde{\Xi}_i^2 +\widetilde{\Omega}^2_i)z^\tau+\widetilde{\Omega}_i(\widetilde{\Xi}_iz^{s_i\tau}+(\widetilde{\Xi}_iz^\tau)^{s_i})\\
           &=(\widetilde{\Xi}_i^2 +\widetilde{\Omega}^2_i)z^\tau+\widetilde{\Omega}_i\bigl(-\epsilon\frac{1}{X_i X^{-1}_{i+1}-1}+\epsilon\frac{1}{X_i X_{i+1}-1}C_i C_{i+1}\bigr)z^{s_i\tau}\\
     &\qquad\qquad+\widetilde{\Omega}_i\biggl(\bigl(-\epsilon\frac{1}{X_i X^{-1}_{i+1}-1}+\epsilon\frac{1}{X_i X_{i+1}-1}C_i C_{i+1}\bigr)z^{\tau}\biggr)^{s_i}\\
     &=(\widetilde{\Xi}_i^2 +\widetilde{\Omega}^2_i)z^\tau+\widetilde{\Omega}_i\bigl(-\epsilon\frac{1}{X_i X^{-1}_{i+1}-1}+\epsilon\frac{1}{X_i X_{i+1}-1}C_i C_{i+1}\bigr)z^{s_i\tau}\\
     &\qquad\qquad+\widetilde{\Omega}_i\bigl(-\epsilon\frac{1}{X^{-1}_i X_{i+1}-1}+\epsilon\frac{1}{X_i X_{i+1}-1}C_{i+1} C_{i}\bigr)z^{s_i\tau}\\
       &=(\widetilde{\Xi}_i^2 +\widetilde{\Omega}^2_i)z^\tau+\widetilde{\Omega}_i(-\epsilon\frac{1}{X_iX^{-1}_{i+1}-1}-\epsilon\frac{1}{X^{-1}_iX_{i+1}-1})z^{s_i\tau}\\
       & =(\widetilde{\Xi}_i^2 +\widetilde{\Omega}^2_i)z^\tau+\epsilon\widetilde{\Omega}_iz^{s_i\tau}\\
       &=z^\tau+\epsilon(\widetilde{\Xi}_i)z^\tau+\epsilon\widetilde{\Omega}_iz^{s_i\tau},
     \end{align*}
     where the last equality is due to $\widetilde{\Xi}_i^2+\widetilde{\Omega}^2_i=(1-\epsilon^2\frac{1}{X_iX^{-1}_{i+1}-1})
     +\epsilon^2\frac{1}{X_iX_{i+1}-1}C_iC_{i+1}$. Thus  $T_i^2z^\tau=\epsilon T_i z^\tau+z^\tau$ in this case. If $s_i$ is not admissible with respect to $\tau\cdot \res(\mathfrak{t}^{\undla})$. This implies that $i,i+1$ are adjacent in tableau $\tau\cdot \mathfrak{t}^{\undla}$ and then $\iota_i=q^2\iota_{i+1}$ or $\iota_i=q^{-2}\iota_{i+1}$. By \eqref{Xi-ztau}, \eqref{substitute}, \eqref{invertible2} and \eqref{resNon-dege-1}, we know that the pair of eigenvalues $(a_i, a_{i+1})$ of $X_i,X_{i+1}$ on $z^\tau$ satisfies \eqref{invertible}, or equivalently, $\widetilde{\Omega}^2_iz^\tau=0$. This together with \eqref{actionformulaNon-dege} lead to 
     \begin{align*}
              T^2_iz^{\tau} =&\widetilde{\Xi}_i^2 z^{\tau}\\
              &=\epsilon^2\bigl(\frac{1}{(X_iX^{-1}_{i+1}-1)^2}
              -\frac{1}{(X_iX_{i+1}-1)(X^{-1}_iX^{-1}_{i+1}-1)}\bigr)z^{\tau}\\
              &\qquad \qquad +\epsilon^2\frac{1}{X_iX_{i+1}-1}C_iC_{i+1} z^{\tau} \\
                &=\epsilon T_i z^\tau+z^\tau-\widetilde{\Omega}^2_iz^\tau\\
                &=\epsilon T_i z^\tau+z^\tau. 
             \end{align*} 
             Hence $T_i^2z^\tau=\epsilon T_iz^\tau+z^\tau$ for each $1\leq i\leq n-1$. 
             
    \medskip          
  Next, we shall check $T_iT_{i+1}T_iz^\tau=T_{i+1}T_iT_{i+1}z^\tau$ {\color{black} for $1\leq i\leq n-2$}. 
 \medskip
 
  {\bf Case I:} $a_i=a^{\pm 1}_{i+2}$. This means $\mathtt{q}(\iota_i)=\mathtt{q}(\iota_{i+2})$. Then by Proposition \ref{separate formula} since $P^{(\bullet)}_n(q,\underline{Q})\neq 0$  and Lemma \ref{important condition1}, the numbers $i$ and $i+2$ lie on the same diagonal and hence $i,i+1,i+2$ must be located in $\tau\cdot \mathfrak{t}^{\undla}$ as the following way: 
 
\[\begin{matrix}
   i & i+1\\
    & i+2.
 \end{matrix}\]
  That is, either $r=\deg{f}$ is odd and $i,i+1,i+2$ are in the $0$-th component or $r=\deg{f}$ is even and $i,i+1,i+2$ are in the $0_-$-th component or $0_+$-component. In both cases, we have either $\iota_i=\iota_{i+2}=1,\,\iota_{i+1}=q^2$ and $a_i=a_{i+2}=1$ or $\iota_i=\iota_{i+2}=-1,\,\iota_{i+1}=-q^2$ and $a_i=a_{i+2}=-1$. Then it's easy to show $T_iT_{i+1}T_i z^\tau=\widetilde{\Xi}_i\widetilde{\Xi}_{i+1}\widetilde{\Xi}_i z^\tau=\widetilde{\Xi}_{i+1}\widetilde{\Xi}_i\widetilde{\Xi}_{i+1} z^\tau=T_{i+1}T_iT_{i+1} z^\tau$ holds by a direct computation.\\
  \medskip
  {\bf Case II:} $a_i\neq a^{\pm 1}_{i+2}$. Set $\widehat{T}_iz^\tau=T_iz^\tau-\widetilde{\Xi}_{i}z^\tau$ for $1\leq i\leq n-1$. It is clear by~(\ref{actionformulaNon-dege}) that
\begin{eqnarray*}
\widehat{T}_i z^{\tau}= \left \{
 \begin{array}{ll}
 \widetilde{\Omega}_i z^{s_i\tau},
 & \text{ if } s_i \text{ is admissible with respect to } \tau\cdot \res(\mathfrak{t}^{\undla}), \\
 0, & \text{ otherwise }.
 \end{array}
 \right.
\end{eqnarray*}
  If $i,i+1$ are adjacent or $i,i+2$ are adjacent, or $i+1,i+2$ are adjacent in $\tau\cdot \mathfrak{t}^{\undla}$, then by \eqref{substitute}, \eqref{invertible2} and \eqref{resNon-dege-1} one can show $\widehat{T}_i\widehat{T}_{i+1}\widehat{T}_iz^\tau
=0=\widehat{T}_{i+1}\widehat{T}_i\widehat{T}_{i+1}z^\tau$.
Otherwise, by~(\ref{Operater2Non-dege}) and \eqref{L-tau-sigma},  we obtain 
\begin{align*}
 \widehat{T}_i\widehat{T}_{i+1}\widehat{T}_iz^{\tau}
&=\sqrt{1-\epsilon^2 \biggl(\frac{a_ia_{i+1}^{-1}}{(a_ia_{i+1}^{-1}-1)^2}+\frac{a_i^{-1}a_{i+1}^{-1}}{(a_i^{-1}a_{i+1}^{-1}-1)^2}\biggr)}\\
&\qquad\sqrt{1-\epsilon^2 \biggl(\frac{a_ia_{i+2}^{-1}}{(a_ia_{i+2}^{-1}-1)^2}+\frac{a_i^{-1}a_{i+2}^{-1}}{(a_i^{-1}a_{i+2}^{-1}-1)^2}\biggr)}\\ 
&\qquad \qquad\qquad \sqrt{1-\epsilon^2 \biggl(\frac{a_{i+1}a_{i+2}^{-1}}{(a_{i+1}a_{i+2}^{-1}-1)^2}+\frac{a_{i+1}^{-1}a_{i+2}^{-1}}{(a_{i+1}^{-1}a_{i+2}^{-1}-1)^2}\biggr)}z^{s_is_{i+1}s_i\tau}\\
&=\widehat{T}_{i+1}\widehat{T}_i\widehat{T}_{i+1}z^{\tau}.
\end{align*}  
Putting together in this case, we have
\begin{align}
\widehat{T}_i\widehat{T}_{i+1}\widehat{T}_iz^{\tau}
=\widehat{T}_{i+1}\widehat{T}_i\widehat{T}_{i+1}z^{\tau}.\label{Braid'}
\end{align}
Since $a_{i}\neq a_{i+1}^{\pm1}, a_{i+1}\neq a_{i+2}^{\pm 1}, a_i\neq a_{i+2}^{\pm 1}$, we have that the element 
\begin{align}\mathsf{Z'}:=& ((X_i+X_i^{-1})-(X_{i+1}+X_{i+1}^{-1}))((X_i+X_i^{-1})-(X_{i+2}+X_{i+2}^{-1}))\cdot\notag\\
&((X_{i+1}+X_{i+1}^{-1})-(X_{i+2}+X_{i+2}^{-1}))\notag
\end{align}
 acts on $z^\tau$ as the non-zero scalar $${((a_i+a_{i}^{-1})-(a_{i+1}+a_{i+1}^{-1}))((a_{i}+a_{i}^{-1})-(a_{i+2}+a_{i+2}^{-1}))((a_{i+1}+a_{i+1}^{-1})-(a_{i+2}+a_{i+2}^{-1}))}. $$
   Recalling the intertwining elements
$\widetilde{\Phi}_i$ from~(\ref{intertwinNon-dege}), we see that
\begin{align}
\widehat{T}_iz^\tau=\widetilde{\Phi}_i \frac{1}{\mathsf{z}_i^2}z^\tau=\widetilde{\Phi}_i \frac{1}{(X_i+X_i^{-1})-(X_{i+1}+X_{i+1}^{-1})}z^\tau\notag
\end{align}
This together with~(\ref{Braidinter}) shows that
\begin{align*}
 \widehat{T}_i\widehat{T}_{i+1} \widehat{T}_iz^\tau
=\widetilde{\Phi}_i\widetilde{\Phi}_{i+1}\widetilde{\Phi}_i \frac{1}{\mathsf{Z'} }z^{\tau},\quad 
 \widehat{T}_{i+1}\widehat{T}_{i} \widehat{T}_{i+1}z^\tau 
=\widetilde{\Phi}_{i+1}\widetilde{\Phi}_{i}\widetilde{\Phi}_{i+1}\frac{1}{\mathsf{Z'}} z^{\tau}
\end{align*}
 Hence by~(\ref{Braid'}) we see that

\begin{align*}
 (\widetilde{\Phi}_i\widetilde{\Phi}_{i+1}\widetilde{\Phi}_i-\widetilde{\Phi}_{i+1}\widetilde{\Phi}_{i}\widetilde{\Phi}_{i+1})
\frac{1}{\mathsf{Z}'}z^{\tau}=0.
\end{align*}
A tedious calculation shows that
\begin{align*}
 \widetilde{\Phi}_i\widetilde{\Phi}_{i+1}\widetilde{\Phi}_i-\widetilde{\Phi}_{i+1}\widetilde{\Phi}_{i}\widetilde{\Phi}_{i+1}&=
(T_iT_{i+1}T_i-T_{i+1}T_iT_{i+1})\mathsf{Z}'. 
\end{align*}
Therefore we obtain $$(T_iT_{i+1}T_i-T_{i+1}T_iT_{i+1})z^\tau=0.
$$

{\bf Relations \eqref{PX1}, \eqref{PX2}, \eqref{PX3}}.  Clearly \eqref{PX3} holds for the action of $T_i$ and $X_j$ on $z^\tau$.  We shall only  check \eqref{PX1}, i.e. $T_iX_iz^\tau=X_{i+1}T_iz^\tau-\epsilon(X_{i+1}+C_iC_{i+1}X_k)z^\tau$ for $1\leq i\leq n-1$ and  the computation for \eqref{PX2} is similar and will be omitted. Observe that $T_iX_i z^\tau=a_iT_i z^\tau$.
      If $s_i $ is admissible with respect to $\tau\cdot \res(\mathfrak{t}^{\undla})$, then by \eqref{Xi-ztau} and \eqref{X-z-tau} we have $$
         T_iX_iz^\tau = a_iT_i z^\tau =a_i(-\frac{1}{a_ia_{i+1}^{-1}-1}+\frac{1}{a_i^{-1}a_{i+1}^{-1}-1}C_iC_{i+1})z^\tau
            +a_i\widetilde{\Omega}_iz^{s_i\tau}.$$ and 
            \begin{align*}
            X_{i+1}T_i z^\tau & =X_{i+1}\epsilon\bigl(-\frac{1}{a_ia_{i+1}^{-1}-1}+\frac{1}{a_i^{-1}a_{i+1}^{-1}-1}C_{i}C_{i+1}\bigr)z^\tau
            +X_{i+1}(\widetilde{\Omega}_iz^{s_i\tau}) \\
                        &=\epsilon\bigl(-\frac{a_{i+1}}{a_ia_{i+1}^{-1}-1}+\frac{a_{i+1}^{-1}}{a_i^{-1}a_{i+1}^{-1}-1}C_{i}C_{i+1}\bigr)z^\tau
            +a_i\widetilde{\Omega}_iz^{s_i\tau}, 
      \end{align*}
   which leads to 
   \begin{equation}\label{check-relation}
   T_iX_i z^\tau=X_{i+1}T_i z^\tau-\epsilon(X_{k+1}+C_iC_{i+1}X_k)z^\tau
   \end{equation}
   since  $(X_{i+1}+C_iC_{i+1}X_{i})z^\tau=(a_{i+1}+a_iC_iC_{i+1})z^\tau. $
  If $s_i $ is not admissible with respect to $\tau\cdot \res(\mathfrak{t}^{\undla})$, it is straightforward to check {\color{black} that \eqref{check-relation} also holds by a similar and easier calculation which we omit.}
  
   {\bf Relations \eqref{PC}}. It is trivial to show $T_iC_jz^\tau=C_jT_iz^\tau,\quad j\neq i, i+1$. It remains to check $T_iC_{i+1}z^\tau=C_iT_iz^\tau-\epsilon(C_i-C_{i+1})z^\tau$ for $1\leq i\leq n-1$ since the other one is similar. 
   Let $1\leq i\leq n-1$.  If $s_i $ is admissible with respect to $\tau\cdot \res(\mathfrak{t}^{\undla})$, we have  
   $$T_iC_{i+1} z^\tau= \epsilon(-\frac{1}{ab-1}C_{i+1}+\frac{1}{a_i^{-1}a_{i+1}-1}C_{i})z^\tau+\widetilde{\Omega}_i (C_{i+1}z^\tau)^{s_i}$$ and 
    \begin{align*}
                  C_iT_iz^\tau &=C_i\epsilon(-\frac{1}{a_ia_{i+1}^{-1}-1}+\frac{1}{a_i^{-1}a_{i+1}^{-1}-1}C_iC_{i+1})z^\tau
                  +C_i\widetilde{\Omega}_iz^{s_i\tau} \\
                  &=\epsilon(-\frac{1}{a_ia_{i+1}^{-1}-1}C_i+\frac{1}{a_i^{-1}a_{i+1}^{-1}-1}C_{i+1})z^\tau
                  +\widetilde{\Omega}_i(C_{i+1}z^\tau)^{s_i}, 
                \end{align*}
    which implies $T_iC_{i+1} z^\tau=C_iT_i z^\tau-\epsilon(C_i-C_{i+1})z^\tau$ for $1\leq i\leq n-1$ with $s_i $ being admissible with respect to $\tau\cdot \res(\mathfrak{t}^{\undla})$. Meanwhile for $1\leq i\leq n-1$ with  $s_i $ being not admissible with respect to $\tau\cdot \res(\mathfrak{t}^{\undla})$,  it can be proved via a similar and easier calculation as above and we omit the detail. 

Putting together, we obtain that $\mathbb{D}(\undla)$ is a $\mHfcn$-module.
\end{proof}

We shall show that $\mathbb{D}(\undla)$ is an irreducible $\mHfcn$-module. The following Lemma will be useful.

\begin{lem}\label{lem:newbijNon-dege}
{\color{black}Fix $\bullet\in\{\mathsf{0},\mathsf{s},\mathsf{ss}\}$, $\undla\in\mathscr{P}^{\bullet,m}_{n}$ and $\tau\in P(\undla)$.} Suppose $s_i$ is admissible with respect to $\tau \mathfrak{t}^{\undla}$ for some $1\leq i\leq n-1$. Then the action of the intertwining element $\widetilde{\Phi}_i$ on $\mathbb{D}(\undla)$ leads to a bijection from $\mathbb{L}(\res(\mathfrak{t}^{\undla}))^{\tau}$ to $\mathbb{L}(\res(\mathfrak{t}^{\undla}))^{s_i\tau}$.
\end{lem}

\begin{proof} First, by Lemma \ref{lem:admissible-residue} and the assumption that  $s_i$ is admissible with respect to $\tau \mathfrak{t}^{\undla}$, we have that $s_i$ is admissible to $\tau\cdot\res(\mathfrak{t}^{\undla})$. Then 
by \eqref{eq:zi}, \eqref{intertwinNon-dege} and \eqref{Operater1Non-dege}, we have 
\begin{equation}\label{eq:Phiz}
\widetilde{\Phi}_i z=\mathsf{z}^2_i(T_i-\widetilde{\Xi}_i)z\in \mathbb{L}(\res(\mathfrak{t}^{\undla}))^{s_i\tau}
\end{equation} for any $z\in \mathbb{L}(\res(\mathfrak{t}^{\undla}))^{\tau}$ by \eqref{actionformulaNon-dege}.  Meanwhile by \eqref{Sqinter} we have 
 \begin{align*}
         \widetilde{\Phi}^2_iz&=\mathsf{z}^2_i\bigl(\mathsf{z}^2_i-\epsilon^2 (X^{-1}_i X^{-1}_{i+1}(X_i X_{i+1}-1)^2-X^{-1}_iX_{i+1}(X_i X^{-1}_{i+1}-1)^2)\bigr)\\
         &=\mathsf{z}^4_i\biggl(1-\epsilon^2\bigl(\frac{X_iX^{-1}_{i+1}}{(X_iX^{-1}_{i+1}-1)^2}
         +\frac{X^{-1}_iX^{-1}_{i+1}}{(X^{-1}_iX^{-1}_{i+1}-1)^2}\bigr)\biggr)z,
        \end{align*} for any $z\in \mathbb{L}(\res(\mathfrak{t}^{\undla}))^{\tau}$, which means $\widetilde{\Phi}^2_i$ act as a scalar on $\mathbb{L}(\res(\mathfrak{t}^{\undla}))^{\tau}$. We only need to show this scalar is non-zero. Actually,  the assumption that $s_i$ is admissible with respect to $\tau \mathfrak{t}^{\undla}$ means $i,i+1$ are neither adjacent nor in the same diagonal in $\mathfrak{t}^{\undla}$. {\color{black}By the second part of Lemma \ref{important condition1} and Lemma \ref{lem:action property-2}}, we deduce that $\mathsf{z}_i$ acts as non-zero scalars on $\mathbb{L}(\res(\mathfrak{t}^{\undla}))^{\tau}$. Moreover, by Lemma \ref{lem:action property-1} and Lemma \ref{lem:action property-2}, $1-\epsilon^2\bigl(\frac{X_iX^{-1}_{i+1}}{(X_iX^{-1}_{i+1}-1)^2}
         +\frac{X^{-1}_iX^{-1}_{i+1}}{(X^{-1}_iX^{-1}_{i+1}-1)^2}\bigr)$ acts as non-zero scalars on $\mathbb{L}(\res(\mathfrak{t}^{\undla}))^{\tau}$. This together with \eqref{eq:Phiz} means that 
       the action of  $\widetilde{\Phi}_i$ on $\mathbb{D}(\undla)$ gives rise to a bijection from $\mathbb{L}(\res(\mathfrak{t}^{\undla}))^{\tau}$ to $\mathbb{L}(\res(\mathfrak{t}^{\undla}))^{s_i\tau}$.
\end{proof}

\begin{prop}\label{irreducibleNon-dege}
 Let  $\undla,\,\underline{\mu}\in\mathscr{P}^{\bullet,m}_{n}$ with $\bullet\in\{\mathsf{0},\mathsf{s},\mathsf{ss}\}$.  Then \begin{enumerate}
  \item $\mathbb{D}(\undla)$ is an irreducible $\mHfcn$-module;
  \item $\mathbb{D}(\undla)$ has the same type as $\mathbb{L}(\res(\mathfrak{t}^{\undla}))$;
  \item $\mathbb{D}(\undla)\cong \mathbb{D}(\underline{\mu})$ if and only if $\undla=\underline{\mu}$.
  \end{enumerate}
\end{prop}
\begin{proof}
 (1) Suppose $N$ is a nonzero $\mHfcn$-submodule of $\mathbb{D}(\undla)$. Observe that the action of commuting elements $X^\pm_1,\ldots,X^\pm_n$ is semi-simple (or can be diagonalized simultaneously) on $\mathbb{D}(\undla)$ and hence on $N$.  Hence $N\cap  \mathbb{L}(\res(\mathfrak{t}^{\undla}))^\tau\neq 0$ for some $\tau\in P(\undla)$. Since $\mathbb{L}(\res(\mathfrak{t}^{\undla}))^{\tau}$ is irreducible as $\mathcal{A}_n$-module one can obtain $N\cap  \mathbb{L}(\res(\mathfrak{t}^{\undla}))^\tau= \mathbb{L}(\res(\mathfrak{t}^{\undla}))^\tau$ and hence $\mathbb{L}(\res(\mathfrak{t}^{\undla}))^\tau\subset N$. Moreover by Corollary \ref{bij}  for each $\sigma\neq \tau\in P(\undla)$, there exists $s_{k_1},s_{k_2},\ldots, s_{k_t}$ with $t\geq 1$ such that $s_{k_l}$ is admissible with respect to $s_{k_{l-1}}\cdots s_{k_1} \cdot \tau\mathfrak{t}^{\undla}$ for $1\leq l\leq t-1$. Then Lemma \ref{lem:newbijNon-dege} shows that that the action of $\widetilde{\Phi}_{k_t}\cdots\widetilde{\Phi}_{k_1}$ sends $\mathbb{L}(\res(\mathfrak{t}^{\undla}))^\tau\subset N$ to  $N\cap \mathbb{L}(\res(\mathfrak{t}^{\undla}))^{\sigma}$. This implies $N\cap \mathbb{L}(\res(\mathfrak{t}^{\undla}))^{\sigma}\neq 0$ and hence $\mathbb{L}(\res(\mathfrak{t}^{\undla}))^{\sigma}\subset N$ for any $\sigma\neq \tau\in P(\undla)$. {\color{black}Thus $N=\mathbb{D}(\undla)$ and hence $\mathbb{D}(\undla)$ is irreducible.}
 
(2) Suppose
$\Psi\in\text{End}_{\mHfcn}(\mathbb{D}(\undla))$. Note that for
each $\tau\in P(\undla)$ and $1\leq i\leq n-1$,  if $s_i$
is admissible with respect to $\tau \mathfrak{t}^{\undla}$, then for
any $z\in \mathbb{L}(\res(\mathfrak{t}^{\undla}))$,
\begin{align}
\Psi(\widetilde{\Omega}_iz^{s_i\tau})=\Psi(T_iz^{\tau}-\widetilde{\Xi}_iz^{\tau})
=T_i\Psi(z^{\tau})-\widetilde{\Xi}_i\Psi(z^{\tau}).\label{PsiskNon-dege}
\end{align} 
Since $s_i$ is admissible with respect to $\tau \mathfrak{t}^{\undla}$, {\color{black} we obtain that $i,i+1$} are not adjacent in  $s_i\tau \mathfrak{t}^{\undla}$. Now Lemma \ref{lem:action property-1}  implies that $\widetilde{\Omega}_i$ acts as a non-zero scalar on $\mathbb{L}(\res(\mathfrak{t}^{\undla}))^{s_i\tau}$. This together with \eqref{PsiskNon-dege} means that $\Psi(z^{s_i\tau})$ can be determined by $\Psi(z^{\tau})$. By Corollary \ref{bij} we deduce that  $\Psi$ is uniquely determined by its restriction on $\mathbb{L}(\res(\mathfrak{t}^{\undla}))$  and moreover by \eqref{actionformulaNon-dege} and \eqref{PsiskNon-dege} we have 
\begin{equation}\label{eq:Psizsigma}
\Psi(z^{\sigma})= (\Psi(z))^\sigma. 
\end{equation} for any $\sigma\in P(\undla)$.  Conversely, any $\mathcal{A}_n$-endomorphism
$\psi:\mathbb{L}(\res(\mathfrak{t}^{\undla}))\rightarrow\mathbb{L}(\res(\mathfrak{t}^{\undla}))$ gives rise to 
$\mHfcn$-endomorphism $\oplus_{\tau\in
P(\undla)}\psi^{\tau}:\mathbb{D}(\undla)\rightarrow
\mathbb{D}(\undla)$, where $\psi^{\tau}(z^{\tau})=(\psi(z))^{\tau}$.  This together with \eqref{eq:Psizsigma} gives rise to $\text{End}_{\mHcn}(\mathbb{D}(\undla))\cong\text{End}_{\mathcal{A}_n}(\mathbb{L}(\res(\mathfrak{t}^{\undla})))$. This proves (2). 

    (3) Clearly if $\undla=\underline{\mu}$ then  $\mathbb{D}(\undla)\cong \mathbb{D}(\underline{\mu})$. Conversely, suppose $\Psi: \mathbb{D}(\undla)\cong \mathbb{D}(\underline{\mu})$ is a $\mHfcn$-module isomorphism. Then there exist $\tau\in P(\undla)$ and $\sigma\in P(\underline{\mu})$ such that $\Psi(\mathbb{L}(\res(\mathfrak{t}^{\undla}))^\tau)=\mathbb{L}(\res(\mathfrak{t}^{\underline{\mu}}))^\sigma$. Since the elements  $X_1+X^{-1}_1,\ldots,X_n+X^{-1}_n$ act as scalars on both $\mathbb{L}(\res(\mathfrak{t}^{\undla}))^\tau$ and $\mathbb{L}(\res(\mathfrak{t}^{\underline{\mu}}))^\sigma$ via the $\mathtt{q}$-value sequence of residus  $\mathtt{q}(\res(\tau\cdot \mathfrak{t}^{\undla})$) and $\mathtt{q}(\res(\sigma\cdot \mathfrak{t}^{\underline{\mu}}))$  in the way given in Lemma \ref{lem:action property-2}. Hence $\mathtt{q}(\res(\tau\cdot \mathfrak{t}^{\undla}))= \mathtt{q}(\res(\sigma\cdot \mathfrak{t}^{\underline{\mu}}))$. 
    Then by Lemma \ref{lem:different residues} we have $\undla=\underline{\mu}$.  
\end{proof}

\subsection{Comparing Dimension}
The following formulae are well-known (cf. \cite[(3.26)]{DJM}, \cite[Theorem 3.1]{Sa}) and will be used in our computation.
\begin{lem}\cite[(3.26)]{DJM} \cite[Theorem 3.1]{Sa}\label{lem:formula}
	{\color{black}Let $n,m,t\in \N.$ We have
$$\sum_{\undla\in \mathscr{P}^{m}_{n}} |\Std(\undla)|^2= n!m^n,\qquad \sum_{\xi\in  \mathscr{P}^{\mathtt{s}}_{t}} \bigl( 2^{\frac{t-\ell(\xi)}{2}}|\Std(\xi)|\bigr)^2=t!.$$ 
}
\end{lem}
{\color{black} The following lemma is straightforward.}
\begin{lem}\label{lem:std-num} 
\begin{enumerate}
\item Let $\undla=(\lambda^{(0)},\lambda^{(1)},\ldots,\lambda^{(m)})=(\lambda^{(0)},\underline{\mu})\in\mathscr{P}^{\mathsf{s},m}_n$ with $|\lambda^{(0)}|=t$ and $\underline{\mu}=(\lambda^{(1)},\ldots,\lambda^{(m)})$, then 
\begin{equation*}
|\Std(\undla)|= \binom{n}{t}|\Std(\lambda^{(0)})|\cdot|\Std(\underline{\mu})|. 
\end{equation*}
\item
{\color{black}Let $\undla=(\lambda^{(0_-)},\lambda^{(0_+)},\lambda^{(1)},\ldots,\lambda^{(m)})=(\lambda^{(0_-)},\lambda^{(0_+)},\underline{\mu})\in\mathscr{P}^{\mathsf{ss},m}_n$} with $|\lambda^{(0_-)}|=a, |\lambda^{(0_+)}|=b$ and $\underline{\mu}=(\lambda^{(1)},\ldots,\lambda^{(m)})$, then 
\begin{equation*}
|\Std(\undla)|= \binom{n}{n-a-b} \binom{a+b}{a}|\Std(\lambda^{(0_-)})|\cdot |\Std(\lambda^{(0_+)})|\cdot|\Std(\underline{\mu})|. 
\end{equation*}
\end{enumerate}
\end{lem}

\begin{thm}\label{CompareDim}
{\color{black} Let $\bullet\in\{\mathtt{0},\mathtt{s},\mathtt{ss}\}$ and $\undQ=(Q_1,Q_2,\ldots,Q_m)\in(\mathbb{K}^*)^m$.}  Assume $f=f^{(\bullet)}_{\undQ}$ and  $P^{\bullet}_{n}(q^2,\undQ)\neq 0$. Then  $\mHfcn$ is a (split) semisimple algebra and 
  $$
  \{\mathbb{D}(\undla)|~ \undla\in\mathscr{P}^{\bullet,m}_{n}\}$$ forms a complete set of pairwise non-isomorphic irreducible $\mHfcn$-module. Moreover,  $\mathbb{D}(\undla)$ is of type  $\texttt{M}$ if and only if $\sharp \mathcal{D}_{\undla}$  is even and is of type  $\texttt{Q}$ if and only if $\sharp \mathcal{D}_{\undla}$  is odd.

\end{thm}

\begin{proof} By Proposition \ref{irreducibleNon-dege}, we have that $$
  \{\mathbb{D}(\undla)|~\undla\in \mathscr{P}^{\bullet,m}_{n}\}$$ is a set of pairwise non-isomorphic irreducible $\mHfcn$-module and $\mathbb{D}(\undla)$ has the same type as $\mathbb{L}(\res(\mathfrak{t}^{\undla}))$. Since $P^{\bullet}_{n}(q^2,\undQ)\neq 0$, by Proposition \ref{separate formula} the tuple of parameters $(q,\undQ)$ is separate with respect to any $\undla\in\mathscr{P}^{\bullet,m}_{n}$. Then by Definition \ref{defn:separate}(1)-(2) and Corollary \ref{lem:irrepAn}, we obtain that $\mathbb{L}(\res(\mathfrak{t}^{\undla}))$ is of type  $\texttt{M}$ if and only if $\sharp \mathcal{D}_{\undla}$  is even and is of type  $\texttt{Q}$ if and only if $\sharp \mathcal{D}_{\undla}$  is odd. 
%
  Now it remains to  prove $\mHfcn$ is semisimple. Firstly,  it is known that every simple $\mHfcn$-module is annihilated by the Jacobson radical $\mathcal{J}(\mHfcn)$ of $\mHfcn$ and hence every $$
  \{\mathbb{D}(\undla)|~\undla\in \mathscr{P}^{\bullet,m}_{n}\}$$ is also a set of pairwise non-isomorphic irreducible $\mHfcn/\mathcal{J}(\mHfcn)$-module. On the other hand, we shall compare the dimension of simple modules with the dimension of $\mHfcn$. {\color{black} Recall that $r=\deg{f}=2m,2m+1, 2m+2$ in the case $\bullet=\mathsf{0},\mathsf{s},\mathsf{ss}$ respectively.}

   {\bf Case 1:} $\bullet=\mathsf{0}$. In this case, $\mathscr{P}^{\bullet,m}_{n}=\mathscr{P}^{m}_{n}$ is the set of $m$-partitions of $n$. Hence all of $\mathbb{D}(\undla)$ are of type $\texttt{M}$. By \eqref{eq:dimDla} we have 
    \begin{align*}
     & \sum_{\undla\in\mathscr{P}^m_n} \dim \mathbb{D}(\undla)=\sum_{\undla\in \mathscr{P}^{m}_{n}} (2^n|\Std(\undla)|)^2\\
    &  =\sum_{\undla\in \mathscr{P}^{m}_{n}} 2^{2n}|\Std(\undla)|^2= 2^{2n}n!m^n=2^nr^nn! =\dim \mHfcn,
   \end{align*}
   where the second equation is due to Lemma \ref{lem:formula} and the last equality is due to Lemma \ref{lem:dim-Hf}. Then by Corollary \ref{cor:dim-compare} we obtain that $\mHfcn$ is semisimple.  
   
  {\bf Case 2:} $\bullet=\mathsf{s}$. In this case, $\mathscr{P}^{\bullet,m}_{n}=\mathscr{P}^{\mathsf{s},m}_{n}$. 
      Then by \eqref{diagonal-D}, \eqref{eq:dimDla} and the fact that $\mathbb{D}(\undla)$ is of type  $\texttt{M}$ if and only if $\sharp \mathcal{D}_{\undla}=\ell(\lambda^{(0)})$  is even and is of type  $\texttt{Q}$ if and only if $\sharp \mathcal{D}_{\undla}=\ell(\lambda^{(0)})$  is odd, which has been verified at the beginning, we obtain 
      \begin{align*}
      &\sum_{\undla\in \mathscr{P}^{\mathtt{s},m}_{n}, \sharp \mathcal{D}_{\undla}:\text{ even}}(\dim \mathbb{D}(\undla))^2+ \sum_{\undla\in \mathscr{P}^{\mathtt{s},m}_{n}, \sharp \mathcal{D}_{\undla}:\text{ odd}}\frac{(\dim \mathbb{D}(\undla))^2}{2}\\
     & =\sum_{\undla\in \mathscr{P}^{\mathtt{s},m}_{n}} \bigl(2^{n-\frac{\sharp \mathcal{D}_{\undla}}{2}}|\Std(\undla)|\bigr)^2\\
      &=\sum_{t=0}^{n} \sum_{\lambda^{(0)}\in  \mathscr{P}^{\mathtt{s}}_{t}}\sum_{\underline{\mu}\in  \mathscr{P}^{m}_{n-t}}\biggl(2^{n-\frac{\ell(\lambda^{(0)})}{2}}
      \biggl(\binom{n}{t}|\Std(\lambda^{(0)})||\Std(\underline{\mu})|\biggr)\biggr)^2\\
      &=\sum_{t=0}^{n} \sum_{\lambda^{(0)}\in  \mathscr{P}^{\mathtt{s}}_{t}}\biggl(2^{n-\frac{\ell(\lambda^{(0)})}{2}}
      \bigg(\binom{n}{t}|\Std(\lambda^{(0)})|\biggr)\biggr)^2(n-t)!m^{n-t}\qquad\text{(by Lemma \ref{lem:formula}  )}\\
      &=\sum_{t=0}^{n} \biggl(2^{n-\frac{t}{2}}
      \binom{n}{t}\biggr)^2(n-t)!m^{n-t}\sum_{\lambda^{(0)}\in  \mathscr{P}^{\mathtt{s}}_{t}} \bigl( 2^{\frac{t-\ell(\lambda^{(0)})}{2}}|\Std(\lambda^{(0)})|\bigr)^2\\
      &=\sum_{t=0}^{n} \biggl(2^{n-\frac{t}{2}}
      \binom{n}{t}\biggr)^2(n-t)!m^{n-t}t!\qquad\text{(by Lemma \ref{lem:formula} )}\\
       &=2^{2n}n!\sum_{t=0}^{n}\binom{n}{t}m^{n-t}(1/2)^t\\
       &=2^{2n}n!(m+1/2)^n=2^nr^nn!=\dim \mHfcn,
      \end{align*} where the second equality is due to Lemma \ref{lem:std-num}(1) and the last equality is due to Lemma \ref{lem:dim-Hf} as $r=2m+1$ in the case $\bullet=\mathsf{s}$. Then by Corollary \ref{cor:dim-compare} we obtain that $\mHfcn$ is semisimple.
        
       {\bf Case 3:} $\bullet=\mathtt{ss}$. In this case, $\mathscr{P}^{\bullet,m}_{n}=\mathscr{P}^{\mathtt{ss},m}_{n}$. 
      Again, by \eqref{eq:dimDla}, \eqref{diagonal-D}, Lemma \ref{lem:std-num}(2) and the fact that $\mathbb{D}(\undla)$ is of type  $\texttt{M}$ if and only if $\sharp \mathcal{D}_{\undla}=\ell(\lambda^{(0_-)})+\ell(\lambda^{(0_+)})$  is even and is of type  $\texttt{Q}$ if and only if $\sharp \mathcal{D}_{\undla}=\ell(\lambda^{(0_-)})+\ell(\lambda^{(0_+)})$  is odd, we obtain \begin{align*}
      	&\sum_{\undla\in \mathscr{P}^{\mathtt{ss},m}_{n}, \sharp \mathcal{D}_{\undla}:\text{ even}}(\dim \mathbb{D}(\undla))^2+ \sum_{\undla\in \mathscr{P}^{\mathtt{ss},m}_{n}, \sharp \mathcal{D}_{\undla}:\text{ odd}}\frac{(\dim \mathbb{D}(\undla))^2}{2}\\
      	&=\sum_{\undla\in \mathscr{P}^{\mathtt{ss},m}_{n}} \bigl(2^{n-\frac{\sharp \mathcal{D}_{\undla}}{2}}|\Std(\undla)|\bigr)^2\\
      &=\sum_{\substack{a,b,c\in\Z_{\geq 0}\\a+b+c=n }} \sum_{\substack{\lambda^{(0_-)}\in \mathscr{P}^{\mathtt{s}}_{a}\\ \lambda^{(0_+)}\in \mathscr{P}^{\mathtt{s}}_{b}}}\sum_{\underline{\mu}\in  \mathscr{P}^{m}_{c}}\biggl(2^{n-\frac{\sharp \mathcal{D}_{\undla}}{2}}
      	\binom{n}{c}\binom{n-c}{a}|\Std(\lambda^{(0_-)})|\cdot|\Std(\lambda^{(0_+)})|\cdot|\Std(\underline{\mu})|\biggr)^2\\
	&=\sum_{\substack{a,b,c\in\Z_{\geq 0}\\a+b+c=n }} 2^{2n}\binom{n}{c}^2\binom{n-c}{a}^2\cdot\\
	&\Big(\sum_{\substack{\lambda^{(0_-)}\in \mathscr{P}^{\mathtt{s}}_{a}}}2^{-\frac{\ell(\lambda^{(0_-)})}{2}}
      	|\Std(\lambda^{(0_-)})|\Big)\cdot\Big( \sum_{\substack{\lambda^{(0_+)}\in \mathscr{P}^{\mathtt{s}}_{b}}}2^{-\frac{\ell(\lambda^{(0_+)})}{2}}
      	|\Std(\lambda^{(0_-)})|\Big)\cdot\Big( \sum_{\substack{\underline{\mu}\in \mathscr{P}^{m}_{c}}}|\Std(\underline{\mu})|\Big)\\
	&=\sum_{\substack{a,b,c\in\Z_{\geq 0}\\a+b+c=n }} 2^{2n}\binom{n}{c}^2\binom{n-c}{a}^2 \frac{a!}{2^a}\frac{b!}{2^b}c!m^c \qquad\text{ (by Lemma \ref{lem:formula} )}\\
	&=\sum_{\substack{a,b,c\in\Z_{\geq 0}\\a+b+c=n }}2^{2n}n!\binom{n}{c}m^c\binom{a+b}{a}2^{-(a+b)}\\
	&=\sum_{c=0}^n2^{2n}n!\binom{n}{c}m^c\sum_{a=0}^{n-c}2^{-n+c}\binom{n-c}{a}\\
	&=\sum_{c=0}^n2^{2n}n!\binom{n}{c}m^c=2^{2n}n!(m+1)^n=2^n(2m+2)^nn!=dim \mHfcn, 
      \end{align*} where the last equality is due to Lemma \ref{lem:dim-Hf} as in this case $r=2m+2$. Then by Corollary \ref{cor:dim-compare} we obtain that $\mHfcn$ is semisimple.
   Hence in all cases, we obtain $ \mHfcn$  is semisimple.  This completes the proof of Theorem.

\end{proof}
 Following \cite{K1, Ru, Wa}, a $\mHfcn$-module $M$ is called completely splittable (or calibrated) if the elements $X_1,X_2,\ldots,X_n$ act semisimply on $M$. Then clearly each $\mathbb{D}(\undla)$ in Theorem \ref{Construction} is completely splittable.  Hence by Theorem \ref{CompareDim} we have the following. 
\begin{cor}\label{splittable}
Let $\undQ=(Q_1,Q_2,\ldots,Q_m)\in(\mathbb{K}^*)^m$.  Assume $f=f^{(\bullet)}_{\undQ}$ with $\bullet\in\{\mathtt{0},\mathtt{s},\mathtt{ss}\}$. If $P^{\bullet}_{n}(q^2,\undQ)\neq 0$,  
then every irreducible $\mHfcn$-module is completely splittable. 
\end{cor}

It is natural to ask whether $P^{\bullet}_{n}(q^2,\undQ)\neq 0$ is the necessary condition for $\mHfcn$ to be semisimple. In fact, we have the following conjecture. 
\begin{conjecture} \label{conjecture}
Let $\undQ=(Q_1,Q_2,\ldots,Q_m)\in(\mathbb{K}^*)^m$.  Assume $f=f^{(\bullet)}_{\undQ}$ with $\bullet\in\{\mathtt{0},\mathtt{s},\mathtt{ss}\}$.The following are equivalent:
\begin{enumerate}
\item The algebra $\mHfcn$ is semisimple. 
\item Every irreducible $\mHfcn$-module is completely splittable.  
\item $P^{\bullet}_{n}(q^2,\undQ)\neq 0$. 
\end{enumerate}
\end{conjecture}

\begin{rem}  In \cite{Wa}, an explicit combinatorial construction of all non-isomorphic irreducible completely splittable modules in the degenerate situation has been obtained. An analogous construction for the non-degenerate case is expected to exist. One possible way to prove the conjecture in the case  $f(X_1)=X_1-1$ is to compare the two sets of partitions parametrizing the isomorphic classes of the irreducible modules given in \cite{BK1} with the set of strict partitions in \cite{Wa}. We  will work on it as well as general cases in a separate work. 

\end{rem}

\begin{cor}
Let $\undQ=(Q_1,Q_2,\ldots,Q_m)\in(\mathbb{K}^*)^m$.  Assume $f=f^{(\bullet)}_{\undQ}$ with $\bullet\in\{\mathtt{0},\mathtt{s},\mathtt{ss}\}$. If $P^{\bullet}_{n}(q^2,\undQ)\neq 0$. Then the center of $\mHfcn$ consists of symmetric polynomials in $X_1+X^{-1}_1,\cdots,X_n+X^{-1}_n$ with dimension $\sharp\mathscr{P}^{\mathtt{\bullet},m}_{n}$.
\end{cor}

\begin{proof}
{\color{black}By Theorem \ref{CompareDim}, $\mHfcn$ is semisimple. Hence the dimension of the center equals to $\sharp\mathscr{P}^{\mathtt{\bullet},m}_{n}$ which is the number of simple modules. Clearly by Lemma \ref{lem:affine-center} the symmetric polynomials in $X_1+X^{-1}_1,\cdots,X_n+X^{-1}_n$ belong to the center of $\mHfcn$.  On the other hand, it is known that  For $\undla\neq\underline{\mu}\in\mathscr{P}^{\mathtt{\bullet},m}_{n}$, by Lemma \ref{lem:eigen-Xk} and Lemma \ref{lem:different residues} that the two multisets of eigenvalues of $X_1+X^{-1}_1,\cdots,X_n+X^{-1}_n$ on $\mathbb{D}(\undla)$ and $\mathbb{D}(\underline{\mu})$ are different. Hence there exists an elementary symmetric polynomial $e_{\undla,\underline{\mu}}$ in $X_1+X^{-1}_1,\cdots,X_n+X^{-1}_n$  such that the eigenvalue of the action of $e_{\undla,\underline{\mu}}(X_1+X^{-1}_1,\cdots,X_n+X^{-1}_n)$ on  $\mathbb{D}(\undla)$ and $\mathbb{D}(\underline{\mu})$ are different. Now we define $$g_{\undla}=\prod_{\underline{\mu}\in\mathscr{P}^{\mathtt{\bullet},m}_{n},\underline{\mu}\neq \undla}(e_{\undla,\underline{\mu}}-e_{\undla,\underline{\mu}}(a^1_{\underline{\mu}},\cdots,a^n_{\underline{\mu}})),
$$ where $a^i_{\underline{\mu}}=\mathtt{q}(\res_{\mathfrak{t}^{\underline{\mu}}}(i))$ is the eigenvalue of $X_i+X^{-1}_i$ acting on $\mathbb{L}(\res(\mathfrak{t}^{\underline{\mu}}))\subset\mathbb{D}(\underline{\mu})$ for $1\leq i\leq n$. Then $g_{\undla}(X_1+X^{-1}_1,\cdots,X_n+X^{-1}_n)$ is a symmetric polynomial in $X_1+X^{-1}_1,\cdots,X_n+X^{-1}_n$ and it acts on $\mathbb{D}(\undla)$ as a non-zero scalar and on $\mathbb{D}(\underline{\mu})$ as $0$ for $\underline{\mu}\neq\undla$. This implies that $\{g_{\undla}(X_1+X^{-1}_1,\cdots,X_n+X^{-1}_n)|\undla\in \mathscr{P}^{\mathtt{\bullet},m}_{n}\}$ is linearly independent in the center of $\mHfcn$ and hence forms a basis of the center of $\mHfcn$.}
\end{proof} 


{\color{black}	\begin{rem}Now we consider  generic non-degenerate cyclotomic Hecke-Clifford algebra, that is, $q,Q_1,\ldots,Q_m$ are algebraically independent over $\Z$. Let $\mathbb{F}$ is the algebraic closure of $\mathbb{Q}(q,Q_1,\ldots,Q_m)$, and  $f=f^{(\bullet)}_{\undQ}$ with $\bullet\in\{\mathtt{0},\mathtt{s},\mathtt{ss}\}$. Then $\mathcal{H}^{f}_{\Delta}(n)$ is a semsimple algebra over $\mathbb{F}$ by Example \ref{deform} and Theorem \ref{CompareDim}.
	\end{rem}
	}
%
%

\section{Degenerate Case}
\subsection{Affine Sergeev algebra}
For $n\in\mathbb{Z}_+$, the affine Sergeev (or degenerate Hecke-Clifford) algebra $\mhcn$ is
the superalgebra generated by even generators
$s_1,\ldots,s_{n-1},$ $x_1,\ldots,x_n$ and odd generators
$c_1,\ldots,c_n$ subject to the following relations
\begin{align}
s_i^2=1,\quad s_is_j =s_js_i, \quad
s_is_{i+1}s_i&=s_{i+1}s_is_{i+1}, \quad|i-j|>1,\label{braid}\\
x_ix_j&=x_jx_i, \quad 1\leq i,j\leq n, \label{poly}\\
c_i^2=1,c_ic_j&=-c_jc_i, \quad 1\leq i\neq j\leq n, \label{clifford}\\
s_ix_i&=x_{i+1}s_i-(1+c_ic_{i+1}),\label{px1}\\
s_ix_j&=x_js_i, \quad j\neq i, i+1, \label{px2}\\
s_ic_i=c_{i+1}s_i, s_ic_{i+1}&=c_is_i,s_ic_j=c_js_i,\quad j\neq i, i+1, \label{pc}\\
x_ic_i=-c_ix_i, x_ic_j&=c_jx_i,\quad 1\leq i\neq j\leq n.
\label{xc}
\end{align}

For $\alpha=(\alpha_1,\ldots,\alpha_n)\in\mathbb{Z}_+^n$ and
$\beta=(\beta_1,\ldots,\beta_n)\in\mathbb{Z}_2^n$, set
$x^{\alpha}=x_1^{\alpha_1}\cdots x_n^{\alpha}$ and
$c^{\beta}=c_1^{\beta_1}\cdots c_n^{\beta_n}$. 
{\color{black} Analogous to the non-degenerate case, we also have the following as mentioned in \cite[Section 2-k]{BK1}. 
\begin{lem}\cite[Theorem 2.2-2.3]{BK1}\label{lem:PBW}
(1) The set $\{x^{\alpha}c^{\beta}w~|~ \alpha\in\mathbb{Z}_+^n,
\beta\in\mathbb{Z}_2^n, w\in {\mathfrak{S}_n}\}$ forms a basis of $\mhcn$.

(2) The center of $\mhcn$ consists of symmetric polynomials in $x_1^2,x_2^2,\ldots,x_n^2$. 
\end{lem}

}
Let $\mathcal{P}_n$ be the superalgebra generated by even generators
$x_1,\ldots,x_n$ and odd generators $c_1,\ldots,c_n$ subject to the
relations~(\ref{poly}),~(\ref{clifford}) and~(\ref{xc}). By
Lemma~\ref{lem:PBW}, $\mathcal{P}_n$ can be identified with the subalgebra
of $\mhcn$ generated by $x_1,\ldots,x_n$ and $c_1,\ldots,c_n$. 

We can define an equivalent relation $\approx$ on $\mathbb{K}$ by $x\approx y$ if and only if $x=y$ or $x+y+1=0$. Let $\mathbf{K}$ be the subset of $\mathbb{K}$  which contains exactly one representative element of $\approx$ and contains $0$ (hence $-1$ is excluded).

For any  $\iota \in \mathbb{K}$, we set
\begin{align}
\mathtt{q}(\iota)=\iota(\iota+1).\label{qi}
\end{align} For $\iota_1\neq \iota_2\in \mathbf{K}$, we have $\mathtt{q}(\iota_1)\neq \mathtt{q}(\iota_2)$.

{\color{black} For any  $x \in \mathbb{K}$, denote by $L(x)$ the $2$-dimensional
$\mathcal{P}_1$-module with
$L(x)_{\bar{0}}=\mathbb{K}v_0$ and $L(x)_{\bar{1}}=\mathbb{K}v_1$
and
$$
x_1v_0=\sqrt{\mathtt{q}(x)}v_0,\quad x_1v_1=-\sqrt{\mathtt{q}(x)}v_1, \quad
c_1v_0=v_1,\quad c_1v_1=v_0.
$$
Clearly $L(x)\cong L(y)$ if and only if $x\sim y\in\mathbb{K}$. Hence 
}for each $\iota\in\mathbf{K}$, we have $2$-dimensional
$\mathcal{P}_1$-module  $L(\iota)$ with 
$L(\iota)_{\bar{0}}=\mathbb{K}v_0$ and $L(\iota)_{\bar{1}}=\mathbb{K}v_1$
and
$$
x_1v_0=\sqrt{\mathtt{q}(\iota)}v_0,\quad x_1v_1=-\sqrt{\mathtt{q}(\iota)}v_1, \quad
c_1v_0=v_1,\quad c_1v_1=v_0.
$$

\begin{lem} The $\mathcal{P}_1$-module $L(\iota)$ is irreducible of type $\texttt{M}$ if $\iota\neq 0$,
and irreducible of type $\texttt{Q}$ if $i=0$. Moreover, $\{L(\iota)|~\iota\in\mathbf{K}\}$ is a complete set of pairwise non-isomorphic finite dimensional irreducible
$\mathcal{P}_1$-module.
\end{lem}
Observe that
$$\mathcal{P}_n\cong \mathcal{P}_1\otimes\cdots\otimes
\mathcal{P}_1.$$ 
{\color{black} For each $\underline{a}=(a_1,a_2,\ldots,a_n)\in(\mathbb{K}^*)^n$, set 
\begin{equation}\label{L-under-a}
L(\underline{a})=L(a_1)\circledast L(a_2)\circledast\cdots\circledast L(a_n), 
\end{equation}
 then $L(\underline{a})\cong L(\underline{b})$ if and only if $a_i\sim b_i$ for $1\leq i\leq n$. }By Lemma~\ref{tensorsmod}, we have the following
result which can be viewed as a generalization of \cite[Lemma 4.8]{BK1}.
\begin{cor} \label{lem:irrepPn}
The $\mathcal{P}_n$-modules
$$
\{L(\underline{\iota})=L(\iota_1)\circledast
L(\iota_2)\circledast\cdots\circledast
L(\iota_n)|~\underline{\iota}=(\iota_1,\ldots,\iota_n)\in\mathbf{K}^n\}
$$
forms a complete set of pairwise non-isomorphic  finite dimensional irreducible
$\mathcal{P}_n$-module.
Moreover, denote by $\gamma_0$ the number of $1\leq j\leq n$ with
$\iota_j=0$. Then $L(\underline{\iota})$ is of type $\texttt{M}$ if
$\gamma_0$ is even and type $\texttt{Q}$ if $\gamma_0$ is odd.
Furthermore,
$\text{dim}~L(\underline{\iota})=2^{n-\lfloor\frac{\gamma_0}{2}\rfloor}$.
\end{cor}

\begin{rem}\label{rem:Ltau2}
Note that each permutation $\tau\in {\mathfrak{S}_n}$ defines a superalgebra
isomorphism $\tau:\mathcal{P}_n\rightarrow \mathcal{P}_n$ by mapping $x_k$ to
$x_{\tau(k)}$ and $c_k$ to  $c_{\tau(k)}$, for $1\leq k\leq n$. For
$\underline{\iota}\in\mathbf{K}^n$, the twist of the action of
$\mathcal{P}_n$ on $L(\underline{\iota})$ with
$\tau^{-1}$ leads to a new $\mathcal{P}_n$-module denoted by
$L(\underline{\iota})^{\tau}$ with
$$
L(\underline{\iota})^{\tau}=\{z^{\tau}~|~z\in L(\underline{\iota})\} ,\quad
fz^{\tau}=(\tau^{-1}(f)z)^{\tau}, \text{ for any }f\in
\mathcal{P}_n, z\in L(\underline{\iota}).
$$
So in particular we have $$(x_kz)^{\tau}=x_{\tau(k)}z^{\tau},\,(c_kz)^{\tau}=c_{\tau(k)}z^{\tau}$$ for each $1\leq k\leq n$. It is easy to see that $
L(\underline{\iota})^{\tau}\cong L(\tau\cdot \underline{\iota})$. Moreover, it is straightforward to show that the following holds 
$$
	(L(\underline{\iota})^\tau)^\sigma\cong L(\underline{\iota})^{\sigma\tau}. 
$$
\end{rem}

\subsection{Intertwining elements for $\mhcn$}
Given 1$\leq i<n$, we define the intertwining element $\Phi_i$ in $\mhcn$ as follows:
\begin{align}
\Phi_i:=s_i(x_i^2-x^2_{i+1})+(x_i+x_{i+1})+c_ic_{i+1}(x_i-x_{i+1}),\quad
.\label{intertw}
\end{align}
These elements satisfy the following properties (cf.\cite[Proposition 3.2, (3.3),(3.4)]{N2}).
\begin{align}
\Phi_i^2=2(x_i^2+x^2_{i+1})-(x_i^2-x^2_{i+1})^2\label{sqinter},\\
\Phi_ix_i=x_{i+1}\Phi_i, \Phi_ix_{i+1}=x_i\Phi_i,
\Phi_ix_l=x_l\Phi_i\label{xinter},\\
\Phi_ic_i=c_{i+1}\Phi_i, \Phi_ic_{i+1}=c_i\Phi_i,
\Phi_ic_l=c_l\Phi_i\label{cinter},\\
\Phi_j\Phi_i=\Phi_i\Phi_j,
\Phi_i\Phi_{i+1}\Phi_i=\Phi_{i+1}\Phi_i\Phi_{i+1}\label{braidinter}
\end{align}
for all admissible $j,i,l$ with $l\neq i, i+1$ and $|j-i|>1$.

{\color{black} For any pair of $(x,y)\in \mathbb{K}^2$, we consider the following condition 
\begin{align}\label{invertible dege}
	(x+y)^2+(x-y)^2=(x^2-y^2)^2.
\end{align}
According to \cite{N1}, via the substitution\begin{align}\label{substitute dege}
	x^2=u(u+1),\qquad\qquad y^2=v(v+1)
\end{align}  the condition \eqref{invertible dege} is equivalent to the condition which states that $u,v$ satisfy one of the following four equations \begin{align}\label{invertible dege2}
	u-v=\pm 1,\quad u+v=0,\quad u+v=-2.
\end{align}
}
\subsection{Cyclotomic Sergeev algebra $\mhgcn$}
As before, to define the cyclotomic Sergeev algebra $\mhgcn$, we need to fix a polynomial $g=g(x_1)\in \mathbb{K}[x_1]$ satisfying  \cite[3-e]{BK1}. Since we are working over algebraically closed field  $\mathbb{K}$,  it is straightforward to check that $g=g(x_1)\in \mathbb{K}[x_1]$ satisfying  \cite[3-e]{BK1} must be one of the following two forms:
$$\begin{aligned}
	g^{(\mathsf{0})}_{\underline{Q}}&=\prod_{i=1}^m\biggl(x^2_1-\mathtt{q}(Q_i)\biggr);\\
	g^{\mathsf{(s)}}_{\underline{Q}}&=x_1\prod_{i=1}^m\biggl(x^2_1-\mathtt{q}(Q_i)\biggr),
\end{aligned}
$$ where $Q_1,\cdots, Q_m\in\mathbb{K}$.

The cyclotomic Sergeev algebra (or degenerate cyclotomic Hecke-Clifford algebra) $\mhgcn$ is defined as $$\mhgcn:=\mhcn/\mathcal{J}_g,
$$ where $\mathcal{J}_g$ is the two sided ideal of $\mhcn$ generated by $g$. Agian, we shall denote the image of $x^{\alpha}, c^{\beta}, w$ in the cyclotomic quotient $\mhgcn$ still by the same symbol. Then we have the following due to \cite{BK1}. 

\begin{lem}\cite[Theorem 3.6]{BK1}
	The set $\{x^{\alpha}c^{\beta}w~|~ \alpha\in\{0,1,\cdots,r-1\}^n,
	\beta\in\mathbb{Z}_2^n, w\in {\mathfrak{S}_n}\}$ forms a basis of $\mhgcn$, where $r=\deg(g)$. 
\end{lem}

From now on, we fix $m\geq 0$ and $Q_1,\cdots, Q_m\in\mathbb{K}$ and let $g=g^{(\mathsf{0})}_{\underline{Q}}$ or $g=g^{(\mathsf{s})}_{\underline{Q}}$. Set $r=\deg(g)$. Then
 $$\begin{aligned}
 g=\begin{cases}
      g^{(\mathsf{0})}_{\underline{Q}}=\prod\limits_{i=1}^m\biggl(x^2_1-\mathtt{q}(Q_i)\biggr) , & \mbox{if $r=2m$ };  \label{res-dege-1}\\
     g^{(\mathsf{s})}_{\underline{Q}}=x_1\prod\limits_{i=1}^m\biggl(x^2_1-\mathtt{q}(Q_i)\biggr), & \mbox{if $r=2m+1$}\label{res-dege-2}.
    \end{cases} 
    \end{aligned}$$

We set $Q_0=0$. 
\begin{defn}Suppose  $\undla\in\mathscr{P}^{\bullet,m}_{n}$ with $\bullet\in\{\mathsf{0},\mathsf{s}\}$  and $(i,j,l)\in \undla$, we define the residue of box $(i,j,l)$ in the degenerate case as follows: \begin{equation}\label{eq:deg-residue}
		\res(i,j,l):=Q_l+j-i.
	\end{equation} If $\mathfrak{t}\in \Std(\undla)$ and $\mathfrak{t}(i,j,l)=a$, we set \begin{align}\label{res-dege}
        \res_\mathfrak{t}(a)&:=Q_l+j-i;\\
\res(\mathfrak{t})&:=(\res_\mathfrak{t}(1),\cdots,\res_\mathfrak{t}(n));\\
\mathtt{q}(\res(\mathfrak{t})):=(\mathtt{q}(\res_{\mathfrak{t}}(1)), \mathtt{q}(\res_{\mathfrak{t}}(2)),\ldots, \mathtt{q}(\res_{\mathfrak{t}}(n)))\label{res-dege2}. 
\end{align}

\end{defn}
The following lemma follows directly from \eqref{res-dege-2} and Corollary \ref{lem:irrepPn}. 
\begin{lem}\label{lem:deg-eigen-Xk}
	{\color{black} Let  $\bullet\in\{\mathsf{0},\mathsf{s}\}$ and $\undla\in\mathscr{P}^{\bullet,m}_{n}$.} Suppose $\mathfrak{t}\in\Std(\undla)$. The eigenvalue of $x_k$ acting on the {\color{black} $\mathcal{P}_n$-module  $L(\res(\mathfrak{t}))$} is $\pm \sqrt{\mathtt{q}(\res_{\mathfrak{t}}(k))}$ for each $1\leq k\leq n$. 
	Hence, the eigenvalue of $x^2_k$ acting on the $\mathcal{A}_n$-module $L(\res(\mathfrak{t}))$ is $\mathtt{q}(\res_{\mathfrak{t}}(k))$ for each $1\leq k\leq n$. 
	
\end{lem} 

\subsection{Separate parameters.}
Let $[1,n]:=\{1,2,\ldots,n\}$. In the rest of this section, we shall introduce a separate condition on the choice of the parameters $\underline{Q}$ and $g=g^{(\bullet)}_{\underline{Q}}$ with $\bullet\in\{\mathtt{0},\mathtt{s}\}$ and $r=\deg(g)$.
	
{\color{black}\begin{defn}
	Let $\undQ=(Q_1,\ldots,Q_m)$ and  $\undla\in\mathscr{P}^{\bullet,m}_{n}$ with $\bullet\in\{\mathsf{0},\mathsf{s}\}$. The parameter $\undQ$  is said to be {\em separate} with respect to $\undla$  if for any $\mathfrak{t}\in \undla$, the $\mathtt{q}$-sequence for $\mathfrak{t}$ defined via \eqref{res-dege2} satisfy the following condition:$$
	\textbf{$\mathtt{q}(\res_{\mathfrak{t}}(k))\neq\mathtt{q}(\res_{\mathfrak{t}}(k+1))$} \begin{matrix}\textbf{ for any $k=1,\cdots,n-1.$ }
	\end{matrix}
	$$
\end{defn}
}
 \begin{rem}
	Observe that in the case $m=0$ and $\bullet=\mathsf{s}$, by saying the parameter $\underline{Q}$ is separate with respect to $\undla=(\lambda^{(0)})$ we mean that the 
	residue of the boxes in the strict partition $\lambda^{(0)}$ defined via \eqref{eq:deg-residue} satisfies the above condition. In the following, we shall apply this convention. 
\end{rem}

\begin{lem}\label{separate in residue-dege}
	Let $\undQ=(Q_1,\ldots,Q_m)$ and  $\undla\in\mathscr{P}^{\bullet,m}_{n}$ with $\bullet\in\{\mathsf{0},\mathsf{s}\}$. Then $(\undQ)$ is separate with respect to $\undla$ if and only if  for any $\mathfrak{t}\in \undla$, and $k=1,\cdots,n-1,$ 
$${\color{black}	\res_{\mathfrak{t}}(k)\neq\res_{\mathfrak{t}}(k+1)\text{ and }\res_{\mathfrak{t}}(k)+\res_{\mathfrak{t}}(k+1)\neq -1 
	}$$
\end{lem}

 \begin{rem}\label{rem:PqQ}
	Recall that we assume $q^2\neq 1$ in order to define the notion of non-degenerate affine Hecke-Clifford superalgebras $\mHcn$. It is known that the notion of the degenerate affine Hecke-Clifford algebras $\mhcn$(or affine Sergeev superalgebras) can be viewed analog of $\mHcn$ the situation $q^2=1$ and we will introduce the polynomial $P_n^{(\bullet)}(q^2,\undQ)$ with $q^2=1$ analogous to $P_n^{(\bullet)}(q^2,\undQ)$ defined above Proposition \ref{separate formula}.
\end{rem}

 Recall that $\undQ=(Q_1,\ldots,Q_m)$. Then for any $n\in \N$, parallel to the polynomials $P^{(\bullet)}_{n}(q^2,\undQ)$ in the case $q^2\neq 1$ in Section \ref{cyclotomic}, we define $P^{(\bullet)}_{n}(1,\undQ)=P^{(\bullet)}_{n}(q^2,\undQ)$ with $q^2=1$ as follows:

\begin{align*}
	P^{(\bullet)}_{n}(1,\undQ):=&
		n! \prod\limits_{i=1}^m\biggl(\prod\limits_{t=3-n}^{n-1}\bigl(2Q_i+t\bigr)\prod\limits_{t=1-n}^{n}\bigl(Q_i+t\bigr)\biggr) \cdot\\
		&\prod_{1\leq i<i'\leq m}\biggl(\prod\limits_{t=1-n}^{n-1}\bigl(Q_i-Q_{i'}+t\bigr)\bigl(Q_i+Q_{i'}+t+1\bigr)\biggr)
\end{align*}for $n\in\N$ and $\bullet\in\{\mathsf{0},\mathsf{s}\}$. Again, when $n=1$, the product $\prod\limits_{t=3-n}^{n-1}\bigl(2Q_i+t\bigr)$ is understood to be $1$.

\begin{prop}\label{separate formula dege}
	Let $n\geq 1,\,m\geq 0$,  $\undQ=(Q_1,\ldots,Q_m)$ and $\bullet\in\{\mathsf{0},\mathsf{s}\}$. Then the  parameter $\undQ$ is separate with respect to $\undla$ for any  $\undla\in\mathscr{P}^{\bullet,m}_{n+1}$ if and only if $P_n^{(\bullet)}(1,\undQ)\neq 0$.
\end{prop}
\begin{proof}
	We assume $n>1$. The condition in Lemma \ref{separate in residue-dege} holds for any $\undla\in\mathscr{P}^{\mathsf{0},m}_{n+1}$ if and only if \begin{align*}
		&\bigl(n\bigr)!\neq 0; \quad \bigl(2Q_i+t\bigr)\neq 0,\quad \forall 3-n\leq t\leq n-1,\,1\leq i\leq m; \\    
	& \bigl(2Q_i+2t\bigr)\neq 0,\quad \forall 1-n\leq t\leq n,\,1\leq i\leq m;\\   
		&\bigl(Q_i-Q_{i'}+t\bigr) \neq 0  ,\quad \forall 1-n\leq t\leq n-1,\,1\leq i\neq i'\leq m;   \\
		&\bigl(Q_i+Q_{i'}+t\bigr)\neq 0 ,\quad \forall 2-n\leq t\leq n,\,1\leq i\neq i'\leq m,
	\end{align*} and the condition in Lemma \ref{separate in residue-dege} holds for any $\undla\in\mathscr{P}^{\mathsf{s},m}_{n+1}$ if and only if
	\begin{align*}
	&\bigl(n\bigr)!\neq 0; 2t\neq 0, \quad \forall 1\leq t\leq n;\quad \bigl(2Q_i+t\bigr)\neq 0,\quad \forall 3-n\leq t\leq n-1,\,1\leq i\leq m; \\    
	& \bigl(Q_i+t\bigr)\neq 0,\bigl(2Q_i+2t\bigr)\neq 0,\quad \forall 1-n\leq t\leq n,\,1\leq i\leq m;\\   
	&\bigl(Q_i-Q_{i'}+t\bigr) \neq 0  ,\quad \forall 1-n\leq t\leq n-1,\,1\leq i\neq i'\leq m;   \\
	&\bigl(Q_i+Q_{i'}+t\bigr)\neq 0 ,\quad \forall 2-n\leq t\leq n,\,1\leq i\neq i'\leq m,
	\end{align*}
	The case $n=1$ can be checked similarly by observing that  the range sets for $t$ in some of the inequalities are slightly different.  Then the Proposition follows from a direct computation. 
\end{proof}

Analogous to the non-degenerate case, we have an the following  parallel Lemma.
 
 \begin{lem}\label{important conditionequi2}
Let $\undQ=(Q_1,\ldots,Q_m)$ and $\bullet\in\{\mathsf{0},\mathsf{s}\}$. Suppose $P^{(\bullet)}_{n}(1,\undQ)\neq 0$. Then for any $\undla\in\mathscr{P}^{\bullet,m}_{n}$ and any $\mathfrak{t}\in\Std(\undla)$, we have the following
\begin{enumerate}
  \item   $\mathtt{q}(\res_{\mathfrak{t}}(k))\neq 0$ for $k\notin \mathcal{D}_{\mathfrak{t}}$;  
  \item   $\mathtt{q}(\res_{\mathfrak{t}}(k))\neq \mathtt{q}(\res_{\mathfrak{t}}(k+1))$  for $k=1,\cdots,n-1$;
    \item  $\res_{\mathfrak{t}}(k)$ and $\res_{\mathfrak{t}}(k+1)$ does not satisfy any one of the four equations in \eqref{invertible dege2}{\color{black}if $k,k+1$ are not in the adjacent diagonals of $\mathfrak{t}.$}
\end{enumerate}  \end{lem}
 
 \begin{lem}\label{lem:deg-action property-1}
 	Let $\undQ=(Q_1,Q_2,\ldots,Q_m)\in(\mathbb{K}^*)^m$ and $\bullet\in\{\mathsf{0},\mathsf{s}\}$.  Suppose $P^{\bullet}_{n}(1,\undQ)\neq 0$. Let $\undla\in\mathscr{P}^{\bullet,m}_{n}$, then 
 	any pair of $(a_1,a_2)$ with $a_1,a_2$ being the eigenvalues of $x_{k}$ and $x_{k+1}$ on $L(\res(\mathfrak{t}))$, respectively, does not satisfy \eqref{invertible dege}, for any $\mathfrak{t}\in\Std(\undla)$ and $k,\,k+1$ being not in the adjacent diagonals of $\mathfrak{t}$. 
 \end{lem}

\begin{lem}\label{lem:different residues deg}
	{\color{black}Let $\bullet\in\{\mathsf{0},\mathsf{s}\}$,  $m\geq 0$ and $\undQ=(Q_1,\ldots,Q_m)\in\mathbb{K}^m$. Suppose $P_n^{(\bullet)}(1,\undQ)\neq 0$. }Then for any $\undla,\,\underline{\mu}\in\mathscr{P}^{\bullet,m}_{n},\,\mathfrak{t}\in\Std(\undla),\,\mathfrak{t'}\in\Std(\underline{\mu})$, we have  $\mathtt{q}(\res(\mathfrak{t}))\neq \mathtt{q}(\res(\mathfrak{t}))$ if $\mathfrak{t}\neq \mathfrak{t'}$.
\end{lem}

\begin{example}\label{deform2} 
	When $Q_1,\ldots,Q_m$ are algebraically independent over $\Z$, and $\mathbb{E}$ is the algebraic closure of $\mathbb{Q}(Q_1,\ldots,Q_m)$, i.e., for generic degenerate cyclotomic Sergeev algebra, the separate condition clearly holds by Proposition \ref{separate formula dege}.
\end{example}

\subsection{Construction of Simple modules}
{\color{black}{\bf For this subsection, we shall fix $\bullet\in\{\mathsf{0},\mathsf{s}\}$, the parameter $\undQ=(Q_1,Q_2,\ldots,Q_m)\in\mathbb{K}^m$ and  the polynomial $g=g^{(\bullet)}_{\undQ}$.  Accordingly, we define the residue of boxes in the Young diagram $\undla$ via \eqref{eq:residue} as well as $\res(\mathfrak{t})$ for each $\mathfrak{t}\in\Std(\undla)$ with $\undla\in\mathscr{P}^{\bullet,m}_{n}$ with $m\geq 0$.}}

{\color{black}
Let $\bullet\in\{\mathsf{0},\mathsf{s}\}$ and $\undla\in\mathscr{P}^{\bullet,m}_{n}$. Suppose $\mathfrak{t}\in\Std(\undla)$ and $1\leq l\leq n$. Similar to Definition \ref{defn:admissible}, if $s_l\cdot\mathtt{q}(\res(\mathfrak{t}))=\mathtt{q}(\res(\mathfrak{u}))$ for some $\mathfrak{u}\in \Std(\undla)$ then the simple transposition $s_l$ is said to be admissible with respect to the sequence $\res(\mathfrak{t})$. }Then analogous to Lemma \ref{lem:admissible-residue} if in addition $\undQ$ is separate with respect to $\undla$, then $s_l$ is admissible with respect to $\mathfrak{t}$ if and only if $s_l$ is admissible with respect to $\res(\mathfrak{t})$ for $1\leq l\leq n$.  
Analogous to $\mathbb{D}(\undla)$ in the non-degenerate case, we define the $\mathcal{P}_n$-module
	$$
D(\undla):=\oplus_{\tau\in P(\undla)}L(\res(\mathfrak{t}^{\undla}))^{\tau}. 
	$$



{\bf In the remaining part of this section, we shall assume that  the parameter $\undQ=(Q_1,Q_2,\ldots,Q_m)\in(\mathbb{K})^m$ satisfies $P_n^{(\bullet)}(1,\undQ)\neq 0$ with $\bullet\in\{\mathtt{0},\mathtt{s}\}$.} {\color{black}By Lemma \ref{important conditionequi2}(1) and \eqref{eq:deg-residue}, we deduce $\{k|1\leq k\leq n, (\res_\mathfrak{t^{\undla}}(k))\sim  0\}=\mathcal{D}_{\mathfrak{t}^{\undla}}$} and \begin{equation}\label{diagonal-D-deg}
	\sharp\mathcal{D}_{\mathfrak{t}^{\undla}}=\sharp\mathcal{D}_{\undla}
	=\left\{
	\begin{array}{ll}
		0,&\text{if }\undla=(\lambda^{(1)},\ldots,\lambda^{(m)})\in\mathscr{P}^{\mathsf{0},m}_{n}\\
		\ell(\lambda^{(0)}),&\text{if }\undla=(\lambda^{(0)},\lambda^{(1)},\ldots,\lambda^{(m)})\in\mathscr{P}^{\mathsf{s},m}_{n}
	\end{array}
	\right. 
\end{equation}
Hence, by Corollary \ref{lem:irrepPn} we have 
\begin{equation}\label{eq:dimDla-deg}
	\text{dim}~D(\undla)=
	2^{n-\lfloor\frac{\sharp \mathcal{D}_{\undla}}{2}\rfloor}\cdot |\Std(\undla)|.
\end{equation}

The following is due to  Remark \ref{rem:Ltau2} and Lemma \ref{lem:deg-eigen-Xk}.
\begin{lem}\label{lem:action property-deg-2}
	Let $\undla\in\mathscr{P}^{\bullet,m}_{n}$ with $\bullet\in\{\mathtt{0},\mathtt{s}\}$. The eigenvalue of $x_k$ acting on the $\mathcal{P}_n$-module $L(\res(\mathfrak{t}^{\undla}))^{\tau}$  is $\pm\sqrt{\mathtt{q}(\res_{\tau\cdot\mathfrak{t}^{\underline{\lambda}}}(k))}$ for each $1\leq k\leq n$. 
	Hence, the eigenvalue of $x^2_k$ acting on the $\mathcal{P}_n$-module $L(\res(\mathfrak{t}^{\undla}))^{\tau}$  is $\mathtt{q}(\res_{\tau\cdot\mathfrak{t}^{\underline{\lambda}}}(k))$ for each $1\leq k\leq n$. 
\end{lem}

To define a $\mhgcn$-module structure  on $D(\undla)$, we introduce two operators on $L(\res(\mathfrak{t}^{\undla}))^{\tau}$ for each $\tau\in P(\undla)$ as follows which are similar with the operators in \cite{Wa}: 
\begin{align}
	\Xi_i u&:=-\Big(\frac{x_i+x_{i+1}}{x_i^2-x_{i+1}^2}+c_ic_{i+1}\frac{x_i-x_{i+1}}{x_i^2-x_{i+1}^2}\Big)u,\label{Operater1-dege}\\
	\Omega_i u&:=\Bigg(\sqrt{1-\frac{2(x_i^2+x_{i+1}^2)}{(x_i^2-x_{i+1}^2)^2}}\Bigg)u.\label{Operater2-dege}
\end{align} where $u\in L(\res(\mathfrak{t}^{\undla}))^{\tau}$.
By the second part of Lemma \ref{important conditionequi2} and Lemma \ref{lem:action property-deg-2}, the eigenvalues of $x^2_i$ and $x^2_{i+1}$ on $L(\res(\mathfrak{t}^{\undla}))^{\tau}$ are different {\color{black} and moreover hence the operators $\Xi_i$ and $\Omega_i$  are well-defined on $L(\res(\mathfrak{t}^{\undla}))^{\tau}$ for each $\tau\in P(\undla)$.  Similar to Theorem \ref{Construction} in non-degenerate case, we have the following theorem which can be proved using a similar way as in \cite[Theorem 4.5]{Wa}. }

\begin{thm}
Let $\bullet\in\{\mathsf{0},\mathsf{s}\}$ and $\undQ=(Q_1,\ldots,Q_m)$. Suppose $g=g(x_1)=g^{(\bullet)}_{\undQ}(x_1)$ and  $P_n^{(\bullet)}(1,\undQ)\neq 0$.
  $D(\undla)$ affords a $\mhgcn$-module via
\begin{align}
s_i z^{\tau}= \left \{
 \begin{array}{ll}
 \Xi_i z^{\tau}
 +\Omega_i z^{s_i\tau},
 & \text{ if } s_i \text{ is admissible with respect to } \tau\cdot \res(\mathfrak{t}^{\undla}), \\
 \Xi_iz^{\tau}
 , & \text{ otherwise },
 \end{array}
 \right.\label{actionformula}
\end{align}
 for  $1\leq i\leq n-1, z\in L(\res(\mathfrak{t}^{\undla}))$ and $\tau\in
P(\undla)$.
\end{thm}


The following is an analogue of Lemma \ref{lem:newbijNon-dege}.

\begin{lem}\label{lem:newbij-dege}
 Fix $\bullet\in\{\mathsf{0},\mathsf{s}\}$ and  $\undla\in\mathscr{P}^{\bullet,m}_{n}$. Let $\tau\in P(\undla)$. Suppose $s_i$ is admissible with respect to $\tau \mathfrak{t}^{\undla}$ for some $1\leq i\leq n-1$. Then the action of the intertwining element $\Phi_i$ on $D(\undla)$ leads to a bijection from $L(\res(\mathfrak{t}^{\undla}))^{\tau}$ to $L(\res(\mathfrak{t}^{\undla}))^{s_i\tau}$.
\end{lem}

Hence we can deduce the analogue of Proposition \ref{irreducibleNon-dege} as follows. 
\begin{prop}\label{irreducible-dege}
	Fix $\bullet\in\{\mathsf{0},\mathsf{s}\}$ and let  $\undla,\,\underline{\mu}\in\mathscr{P}^{\bullet,m}_{n}$.  Then \begin{enumerate}
		\item $D(\undla)$ is an irreducible $\mhgcn$-module;
		\item $D(\undla)$ has the same type as $L(\res(\mathfrak{t}^{\undla}))$;
		\item $D(\undla)\cong D(\underline{\mu})$ if and only if $\undla=\underline{\mu}$.
	\end{enumerate}
\end{prop}

%

Finally, we obtain the analogue of Theorem \ref{CompareDim}.

\begin{thm}\label{CompareDim-dege}
	{\color{black}Let $\bullet\in\{\mathtt{0},\mathtt{s}\}$ and $\undQ=(Q_1,Q_2,\ldots,Q_m)\in(\mathbb{K})^m$.  Assume $g=g^{(\bullet)}_{\undQ}$ and   $P_n^{(\bullet)}(1,\undQ)\neq 0$.} Then  $\mhgcn$ is a (split) semisimple algebra and 
	$$
	\{D(\undla)|~ \undla\in\mathscr{P}^{\bullet,m}_{n}\}$$ forms a complete set of pairwise non-isomorphic irreducible $\mhgcn$-module. Moreover,  $D(\undla)$ is of type  $\texttt{M}$ if and only if $\sharp \mathcal{D}_{\undla}$  is even and is of type  $\texttt{Q}$ if and only if $\sharp \mathcal{D}_{\undla}$  is odd.

\end{thm}

\begin{cor}
	{\color{black}Let $\bullet\in\{\mathtt{0},\mathtt{s}\}$ and $\undQ=(Q_1,Q_2,\ldots,Q_m)\in(\mathbb{K})^m$.  Assume $g=g^{(\bullet)}_{\undQ}$ and   $P_n^{(\bullet)}(1,\undQ)\neq 0$.}  Then the center of $\mhgcn$ consists of symmetric polynomials in $x^2_1,\cdots,x^2_n$ with dimension $\sharp\mathscr{P}^{\mathtt{\bullet},m}_{n}$.
\end{cor}

Again, let's consider  generic cyclotomic Sergeev algebra, that is, $Q_1,\ldots,Q_m$ are algebraically independent over $\Z$. Let $\mathbb{E}$ is the algebraic closure of $\mathbb{Q}(Q_1,\ldots,Q_m)$, and  $g=g^{(\bullet)}_{\undQ}$ with $\bullet\in\{\mathtt{0},\mathtt{s}\}$. Then $\mhgcn$ is a semsimple algebra over $\mathbb{E}$ by Example \ref{deform2} and Theorem \ref{CompareDim-dege}. Analogous to Conjecture \ref{conjecture}, we also have a parallel conjecture on the semisimplicity of $\mhgcn$.

\end{document}